\documentclass[11pt]{amsart}
\usepackage{amsmath, amssymb, amscd, mathrsfs, slashed, url,color,extsizes,xcolor,multicol,indentfirst,latexsym,bm,graphicx,subfigure,esint,float,verbatim,comment,epsfig,dsfont,amsthm,amsfonts,amsbsy,tikz-cd,amsthm,thmtools,bbm,adjustbox,enumerate,calligra,fancyhdr,accents,enumitem,etoolbox}

\usepackage{cancel}

\usepackage[all,cmtip]{xy}
\usepackage[top=1in, bottom=1in, left=1.25in, right=1.25in]{geometry}
\usepackage[colorlinks,hypertexnames=false]{hyperref}
\usepackage[toc,page]{appendix}

\DeclareMathAlphabet{\mathcalligra}{T1}{calligra}{m}{n}

\declaretheoremstyle[
headfont=\color{blue}\normalfont\bfseries,
bodyfont=\color{blue}\normalfont\itshape,
]{colored}

\usepackage[utf8]{inputenc}
\usepackage[english]{babel}

\setlist[enumerate]{
    label={(\roman*)},
}

\AtBeginEnvironment{definition}{%
    \setlist[enumerate]{
        label={(\roman*)}, 
        font=\normalfont,
        before=\normalfont
    }%
}

\newtheorem{theorem}{Theorem}[section]

\newtheorem{corollary}[theorem]{Corollary}

\newtheorem{definition}[theorem]{Definition}
\newtheorem{question}[theorem]{Question}

\newtheorem{lemma}[theorem]{Lemma}

\newtheorem{proposition}[theorem]{Proposition}
\newtheorem{remark}[theorem]{Remark}

\theoremstyle{definition}

\newcommand{\RR}{\mathbb{R}}
\newcommand{\NN}{\mathbb{N}}
\newcommand{\ZZ}{\mathbb{Z}}
\newcommand{\Sig}{\Sigma}

\newcommand{\CC}{\mathbb{C}}

\renewcommand{\d}{\operatorname{d}}

\newcommand{\ZT}{\ZZ/2}

\newcommand{\I}{\mathcal{I}}

\newcommand{\Flux}{\operatorname{Flux}}

\usepackage{graphicx} 

\title{Prescribed Singular Sets for $\ZT$-Harmonic $1$-Forms on $\RR^n$}
\author{}

\begin{document}

\author{Jiahuang Chen, Siqi He, Willem Adriaan Salm} 

\address{AMSS}
\email{chenjiahuang@amss.ac.cn, sqhe@amss.ac.cn}	
\address{Oxford University}
\email{andries.salm@maths.ox.ac.uk}

\hyphenation{inver-ti-bi-li-ty}
\maketitle
\begin{abstract}
	$\mathbb{Z}/2$ harmonic $1$-forms arise naturally as singular limits in gauge theory and calibrated geometry, 
	but the topology that can occur in their singular sets is poorly understood. 
	In this paper we study the flexibility of these singular sets when the ambient Riemannian metric is allowed to vary. 
	We prove that every compact smoothly embedded codimension-two submanifold with trivial normal bundle in 
	a Euclidean space of dimension at least three can be realized as the singular set of a 
	nondegenerate $\mathbb{Z}/2$ harmonic $1$-form for a complete metric that is Euclidean outside a compact set.
	As an application, we use Calabi surgery to construct $\mathbb Z/2$-harmonic $1$-forms with prescribed local singular sets on 
	closed manifolds of positive first Betti number.
\end{abstract}

\section{Introduction}

$\mathbb Z/2$-harmonic $1$-forms are singular elliptic objects that arise
naturally as limits of solutions to several nonlinear geometric equations.
They were introduced by Taubes in his compactness theory for
$\mathrm{PSL}_2(\mathbb C)$ connections \cite{Tau13,Tau14}. Closely related
$\mathbb Z/2$-harmonic forms and spinors also appear in compactness problems
for generalized Seiberg--Witten equations and other gauge-theoretic systems
\cite{HW15,WZ21}. Such a form comes equipped with a singular set, which is
part of the geometric data. In this paper, we focus on the case where the
singular set is a smooth codimension-two submanifold.

Although the equation defining $\ZT$-harmonic $1$-forms is linear once the singular set is fixed,
the deformation problem is intrinsically nonlinear because the singular set
itself is allowed to move. Donaldson developed a deformation framework for
$\ZT$-harmonic $1$-forms with smooth singular sets
\cite{Donaldson2021}; see also Parker \cite{Parker2023Deformations} for the
corresponding theory for spinors. For related analytic foundations and
developments in the deformation and gluing theories of $\mathbb Z/2$-harmonic
spinors and $1$-forms, see
\cite{Takahashi2015Thesis, HeParker2024, WalpuskiGora2026,HeParkerWalpuski2026}.

Let $(M,g)$ be an oriented Riemannian manifold, and let $\Sigma\subset M$ be a
closed, cooriented, smoothly embedded codimension-two submanifold. Let $\I\to M\setminus\Sigma$ be a flat real line bundle with monodromy $-1$ around every small normal meridian of $\Sigma$.
For $v \in \Gamma(T^*M \otimes \mathcal{I})$, we say that $(v,\Sigma,\I)$ is an $L^2$-bounded $\ZT$-harmonic 1-form if $v$ is harmonic and locally $L^2$-bounded near $\Sigma$.

Near the branching set $\Sigma$, $v=\d u$ for some $\I$-valued harmonic function $u$. Donaldson \cite{Donaldson2021} showed that, after choosing a suitable complex normal coordinate $z$ near $\Sigma$, $u$ has an expansion of the form
\begin{equation}
	\label{eq:Donaldsonexpansion}
	u
	=
	\operatorname{Re}\!\left(A(u)z^{1/2}+B(u)z^{3/2}\right)
	-\frac12
	\operatorname{Re}\!\left(A(u)z^{1/2}\right)
	\operatorname{Re}(\bar{\mu}z)
	+O(|z|^{5/2}).
\end{equation}
Here $\mu$ denotes the mean-curvature vector of $\Sigma$ in $(M,g)$,
expressed in the chosen complex normal coordinate.

\begin{definition}
	\label{def:nondegenerate-local-model}
	Let $(v,\Sigma,\I)$ be an $L^2$-bounded $\ZT$-harmonic $1$-form whose local potential
	$u$ has the expansion \eqref{eq:Donaldsonexpansion}. We say that
	$(v,\Sigma,\I)$ is $\ZT$-harmonic if $A(u)\equiv 0$. It is
	\textbf{nondegenerate} if, in addition, $B(u)$ is nowhere vanishing along
	$\Sigma$.
\end{definition}

On a closed manifold $M$, the topology of the ambient manifold can impose nontrivial restrictions on the singular sets of $L^2$-bounded $\ZT$-harmonic $1$-forms \cite{Hay22}. This is due to a Hodge theory for $L^2$-bounded $\ZT$-harmonic 1-forms \cite{HunsickerMazzeo2005}. However, there is no Hodge theory for $\ZT$-harmonic 1-forms as the governing differential equation has an infinite dimensional co-kernel. We ask therefore whether the stronger analytical restrictions in
Definition~\ref{def:nondegenerate-local-model} still impose topological restrictions on the smooth singular set.

\begin{question}
	\label{question:topology}
	Let $(v,\Sigma,\I)$ be a nonzero $\ZT$-harmonic $1$-form on $(M,g)$. To what
	extent does the topology of $M$ constrain the topology of $\Sigma$?
\end{question}

Our first result shows that once compactly supported
changes of the ambient metric are allowed, any closed codimension-two submanifold with trivial normal bundle can be used as the branching set of a nondegenerate $\ZT$-harmonic function on $\RR^n$.
Inspired by Salm's long-neck construction \cite{salm2024construction}, we
establish this flexibility while keeping the underlying smooth manifold fixed
and leaving the Euclidean geometry near infinity unchanged. The following
theorem makes this statement precise.

\begin{theorem}
	\label{theorem:euclidean-realization}
	Let $n\geq 3$, and let $\Sigma\subset\mathbb R^n$ be a compact smoothly
	embedded codimension-two submanifold with trivial normal bundle. There
	exist a smooth complete Riemannian metric $g$ on $\mathbb R^n$ that is
	Euclidean outside a compact set, a flat real line bundle
	$\I\to\mathbb R^n\setminus\Sigma$ with monodromy $-1$ around every normal
	meridian of $\Sigma$, and a nonzero harmonic section $u$ of $\I$ such that $du$ is a
	nondegenerate $\ZT$-harmonic $1$-form with singular set $\Sigma$.
\end{theorem}

\begin{remark}
	For the standard Euclidean metric, Li, Mashayekhi, and Zhang recently proved that a $\ZT$-harmonic function with a smooth compact branching set, vanishing $A$-coefficient, and nondegenerate quadratic asymptotics is, up to rigid motions, one of the Donaldson--Yan models \cite{LiMashayekhiZhang2026}. This
	result concerns the standard Euclidean metric. In
	Theorem~\ref{theorem:euclidean-realization}, we allow compactly supported
	perturbations of this metric and thereby realize arbitrary compact smooth
	codimension-two singular sets with trivial normal bundle.
\end{remark}

The construction admits an additional refinement which is essential for
passing to closed manifolds. The realizing section $u$ can be chosen to have
controlled polynomial growth, with an asymptotic harmonic polynomial $P$ at infinity. See Proposition~\ref{prop:phase-family-of-zonal-polynomials} and Theorem~\ref{thm:nash-moser-correction}.
We call any polynomially growing realization a \emph{Euclidean model}.

After choosing a trivialization of $\I$ near infinity, we can define the flux
$$\Flux(u):= \lim_{R \to \infty}\int_{S_R^{n-1}}\partial_r u\,dA$$ for a Euclidean model. 
Theorem~\ref{thm:flux-normalization} shows that the $\ZT$-harmonic 1-form in
Theorem~\ref{theorem:euclidean-realization} can be chosen so that its flux is zero. We call a Euclidean model with vanishing flux a
\emph{zero-flux Euclidean model}.
This zero-flux refinement is precisely the compatibility condition that
allows the Euclidean model to be transplanted into a closed manifold. Polynomial growth, meanwhile, provides the finite-dimensional
asymptotic data needed for matching.

Explicitly, let $(v,\Sigma,\I)$ be a nonzero
$\ZT$-harmonic $1$-form on a closed oriented Riemannian manifold $(M^n,g)$,
let $q\in M\setminus \Sigma$, and let $P_q$ be the first nonzero homogeneous term
of a local harmonic potential for $v$ at $q$.  We call $q$ a \emph{matching point} if
$dP_q$ is nowhere zero on $\mathbb R^n\setminus\{0\}$.  A Euclidean model is said
to match $q$ if the highest-degree part of its asymptotic polynomial agrees
with $P_q$ up to multiplication by a nonzero constant and a
rotation of coordinates. The following theorem is an application of the Calabi-surgery construction in
\cite{HeChenYan}.

\begin{theorem}
	\label{theorem:calabi-insertion}
	Assume that $n\geq3$. Let $q$ be a matching point of
	$(v,\Sigma,\I)$, and let $(u,\Sig_{\RR^n},\I_{\RR^n})$ be a
	zero-flux Euclidean model matching $q$. For every sufficiently small
	$\delta>0$, there are a metric $g_\delta$ and a $\ZT$-harmonic
	$1$-form $(v_\delta,\Sig_\delta,\I_\delta)$ on $(M,g_\delta)$ such that
	\[
		(g_\delta,v_\delta,\I_\delta)=(g,v,\I)
		\quad\text{on }M\setminus B_\delta(q),
		\qquad
		\Sig_\delta\setminus B_\delta(q)=\Sigma.
	\]
	Moreover, $\Sig_\delta\cap B_{\delta/2}(q)$ is a rescaled copy of
	$\Sig_{\RR^n}$.
\end{theorem}

Every regular point is a matching point with a linear
matching polynomial, so the matching condition is automatic for
models of degree one.  Combining such a model with a nonzero ordinary harmonic
$1$-form gives the following topological application, which generalizes
earlier results in \cite{HeChenYan,Yan2025Euclidean}.

\begin{theorem}
	\label{corollary:b1-positive}
	Let $(M^n,g)$ be a closed oriented Riemannian manifold, where $n\geq 3$, and suppose that it carries a non-zero 
	$\ZT$-harmonic $1$-form $(v,\Sig_M,\I_M)$. (The branching set $\Sig_M$ may be singular or even empty.) 
	Let $\Sigma_0\subset\mathbb R^n$ be a compact connected oriented codimension-two submanifold with trivial normal bundle.
	For every $p \in M^n \setminus \Sig_M$ with $v(p)\neq 0$,
	there exist a metric $g'$ and a $\ZT$-harmonic $1$-form on $(M^n,g')$,
	such that the branching set is the union of $\Sig_M$ and a rescaled copy of $\Sigma_0$.
	The new $1$-form is nondegenerate along the inserted copy of $\Sigma_0$.
	The new metric, $1$-form, and flat line bundle agree with the original data on the complement of a neighbourhood of $p$.
\end{theorem}

Note that when $b_1(M) > 0$, any ordinary nonzero harmonic $1$-form can be seen as a $\ZT$-harmonic $1$-form with empty branching set. Hence, such forms can be used as the initial data for Theorem \ref{corollary:b1-positive}. If one looks for examples of $\ZT$-harmonic $1$-forms on manifolds with $b_1(M)=0$, one can consider
\cite{He2025,HeWentworthZhang2024}.

In dimension three, the need for a matching point in
Theorem~\ref{theorem:calabi-insertion} poses no obstruction, even when a matching polynomial of higher degree is required.
Given any regular point $q$ and any integer
$m\geq3$, the Calabi-surgery construction of \cite{HeChenYan} produces a new
$\ZT$-harmonic $1$-form for which $q$ is a matching point whose matching polynomial is the
zonal polynomial $Z_m$; see
Proposition~\ref{prop:zonal-matching-point-preparation}.
As a consequence, singular sets with any prescribed finite number of connected components can be inserted into closed $3$-manifolds; 
see Corollaries~\ref{cor:zonal-zero-flux-model-existence} and \ref{cor:zonal-zero-flux-insertion}. 
Thus, even on closed manifolds, the topology of the singular set is subject to remarkably few restrictions.

\subsection*{Acknowledgments}
We would like to thank Thomas Walpuski for helpful comments. 

\subsection*{AI disclosure}
All mathematical ideas and proofs were developed by the authors. At the same time, AI tools were used in the editing and review of this manuscript. It is also used in figure preparation.

\begin{remark}
	In the proof of Theorem \ref{theorem:euclidean-realization}, we used the Nash--Moser theorem. After finishing this manuscript, we realized there may be an alternative proof that does not require this. 


	
	Namely, Theorem \ref{theorem:calabi-insertion} is written for Euclidean models, but we expect that it also works for cylindrical models. The boundary region of Section \ref{subsec:model-spaces-metric} would be an example of such a cylindrical model.
	Using the Fourier decomposition, all multivalued harmonic functions are known on the boundary region. This is explicitly shown in Equation \eqref{eq:radial-ode}. One can use these model solutions as an alternative for the matching polynomials.
	
	Moreover, we don't expect that there is a flux condition for these cylindrical models. Namely, the flux condition in the Euclidean case is necessary to get rid of a cohomological obstruction. Due to Lemma 3.5 in \cite{HeSalm2026MetricPerturbations}, there won't be a cohomological obstruction in the cylindrical case.
\end{remark}
\section{Long neck spaces and elliptic theory}
\label{sec:long-neck-elliptic}

This section introduces the geometric and analytic framework for the long-neck
spaces, which have been used in \cite{salm2024construction}. We first construct a family of metrics $g_s$ on the complement of a fixed codimension-two set by inserting a cylindrical neck near the branch locus. With these spaces in mind, we set up the weighted H\"older norms and prove uniform regularity estimates.

\subsection{The model spaces}
\label{subsec:model-spaces-metric}

Let $n\ge3$.  In this subsection we construct a family of long-neck metrics
$g_s$ on $\RR^n$ and its limit space  $X_\infty:=(\RR^n\setminus\Sigma, g_\infty)$, following the neck-stretching model of \cite{salm2024construction}. Let $\Sigma\subset\RR^n$ be a fixed compact codimension-two submanifold with trivial normal bundle, and denote its connected components by $\Sigma_a$, i.e. $\Sigma=\coprod_{a=1}^p\Sigma_a$. Let $\mathcal I\to \RR^n\setminus \Sigma$ be the flat real Euclidean line bundle with monodromy $-1$ around every positive normal meridian.

Fix $R_0 \gg 1$ once and for all and let $s>0$. In our model space, this constant will play the role of the length of the neck. By the triviality of the normal bundle and using the exponential map, a tubular neighbourhood of $\Sigma$ can be
identified\footnote{
As we haven't defined a metric on $\RR^n$ yet, this identification is only up to diffeomorphism. This also allows us to pick the ball $B_{R_0+s+1}(0) \subset \CC$ as large as we want, because the coordinate $r$ does not relate to the geodesic distance to $\Sigma$ with respect to the flat metric.} with $B_{R_0+s+1}(0) \times \Sigma$ for the ball $B_{R_0+s+1}(0) \subset \CC$. On $\CC$ we have the coordinates $z = r e^{i \phi}$, which we will use throughout this paper. 

For each $a\in\{1,\ldots,p\}$, we decompose the stretched tubular
neighbourhood of $\Sigma_a$ into two regions:
\begin{align}
\label{eq:finite-neck-regions}
        \mathcal B_{\Sigma_a}
        &:=
        B_{R_0}(0)\times\Sigma_a, \notag\\
        \mathcal N_{\Sigma_a,s}
        &:=
        \bigl(B_{R_0+s}(0)\setminus B_{R_0}(0)\bigr)\times\Sigma_a .
\end{align}
Here $B_R(0)$ denotes the disk of radius $R$ in the normal complex plane.  We
call $\mathcal B_{\Sigma_a}$ the \textbf{boundary region} of $\Sigma_a$,
$\mathcal N_{\Sigma_a,s}$ the \textbf{neck region} of $\Sigma_a$, and $s$ the
\textbf{neck length}.

Choose $R_\infty\gg1$ so that the stretched tubular neighbourhood is contained
in $B_{R_\infty}(0)\subset\RR^n$.  Set
\begin{equation}
\label{eq:asymptotic-transition-regions}
\begin{split}
        \mathcal A_\infty
        &:=
        \RR^n\setminus B_{R_\infty}(0), \\
        \mathcal C_{\Sigma,s}
        &:=
        B_{R_\infty}(0)\setminus
        \bigsqcup_{a=1}^p
        \bigl(\mathcal B_{\Sigma_a}\cup\mathcal N_{\Sigma_a,s}\bigr).
\end{split}
\end{equation}
We call $\mathcal A_\infty$ the \textbf{asymptotic region} and
$\mathcal C_{\Sigma,s}$ the \textbf{transition region}. 
(See Figure \ref{fig:finite-neck-regions} for a schematic overview of all the regions.)

We now define $g_s$ region by region.  Fix a Riemannian metric $h_\Sigma$ on
$\Sigma$, and write $h_{\Sigma_a}:=h_\Sigma|_{\Sigma_a}$.

On each neck region $\mathcal N_{\Sigma_a,s}$, define
\begin{equation}
\label{eq:metric-neck-region}
        g_s=\mathrm d r^2+\mathrm d\phi^2+\varepsilon^2h_{\Sigma_a},
\end{equation}
where the constant $\varepsilon$ will be explicitly chosen in Definition \ref{def:epsilon-in-metric}. For now, it is sufficient to assume that $\varepsilon$ is a small number depending only on $h_\Sigma$.
In summary, the neck is a cylinder with cross-section
$S^1_\phi\times\Sigma_a$.

To extend this metric to the boundary region, choose a smooth nondecreasing
function $\tilde r:[0,\infty)\to[0,1]$ such that $\tilde r(r)=r$ for
$r\le\frac12$ and $\tilde r(r)=1$ for $r\ge R_0$.  On
$\mathcal B_{\Sigma_a}$, define
\begin{equation}
\label{eq:finite-neck-metric-boundary}
        g_s
        =
        \mathrm d r^2+\tilde r(r)^2\,\mathrm d\phi^2
        +\varepsilon^2h_{\Sigma_a}.
\end{equation}
Hence, close to $\Sigma_a$,
$g_s=\mathrm d r^2+r^2\,\mathrm d\phi^2+\varepsilon^2h_{\Sigma_a}$, while near
$r=R_0$ it agrees with the cylindrical metric on the neck.

On $\mathcal A_\infty$, we impose the Euclidean metric $g_s=g_{\mathrm{Eucl}}$.  On
$\mathcal C_{\Sigma,s}$, we choose a smooth interpolation between the
cylindrical metric near the neck and the Euclidean metric near $\mathcal A_\infty$.
We choose these interpolations so that there is a fixed compact Riemannian
manifold $(\mathcal C_\Sigma,g_{\mathcal C})$ and collar-preserving
diffeomorphisms $\Phi_s:\mathcal C_\Sigma\to\mathcal C_{\Sigma,s}$ satisfying
$\Phi_s^*g_s=g_{\mathcal C}$. Thus $\mathcal C_{\Sigma,s}$ may move inside
$\RR^n$, but its intrinsic geometry is independent of $s$.
We write
\[
        X_s:=(X,g_s)=(\RR^n,g_s)
\]
for the finite-neck space.  Thus, the underlying manifold is $\RR^n$, while the geometry depends on $s$.

\begin{figure}[htbp]
        \centering
        \includegraphics[width=0.82\textwidth]{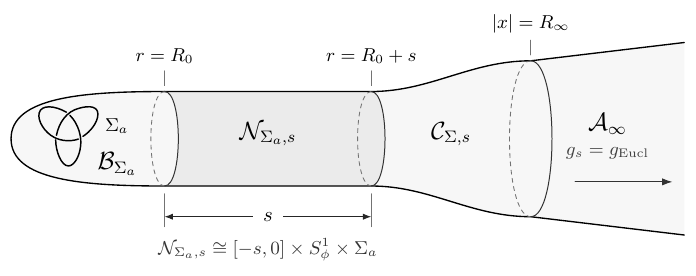}
        \caption{A schematic decomposition of $X_s$.}
        \label{fig:finite-neck-regions}
\end{figure}

Finally, we also set up the limiting space $X_\infty$.

Let $\mathcal{N}_{\Sigma_a, \infty}:=(-\infty, 0] \times S^1_\phi\times\Sigma_a$, where we denote the coordinate on $(-\infty, 0]$ by $t$.
Equip $\mathcal{N}_{\Sigma_a, \infty}$ with the metric
\[
        g_\infty=\mathrm d t^2+\mathrm d\phi^2
        +\varepsilon^2h_{\Sigma_a}.
\]
Using the fixed transition model $\mathcal C_\Sigma$, define
\[
        X_\infty:=\mathcal A_\infty\cup\mathcal C_\Sigma\cup
        \bigcup_{a=1}^p\mathcal{N}_{\Sigma_a, \infty}
\]
Here $\mathcal C_\Sigma$ is glued to $\mathcal A_\infty$ along their common boundary, and $\{0\}\times S^1_\phi\times\Sigma_a$ is glued to the corresponding inner boundary of $\mathcal C_\Sigma$. 
Let $g_\infty$ equal $g_{\mathrm{Eucl}}$ on
$\mathcal A_\infty$, equal $g_{\mathcal C}$ on $\mathcal C_\Sigma$.
This gluing identifies $t=0$ with $r=R_0+s$; equivalently, on the finite neck of $X_s$,
\[
        t=r-(R_0+s),
\]
so that $\mathcal N_{\Sigma_a,s}$ is identified with $[-s,0]\times S^1_\phi\times\Sigma_a$. We call $\mathcal{N}_{\Sigma_a, \infty}$ a \textbf{cylindrical end} of $X_\infty$.

\subsection{H\"older norms}
\label{subsec:definition-holder-norms}
Next we introduce the weighted H\"older spaces we use in this paper.
When considering the model space with a finite length neck, we use the Donaldson-type H\"older norms from
\cite{Donaldson2021} near \(\Sigma\), while on the Euclidean end we use standard
weighted norms as in \cite{McOwen1979}. We first define the finite-neck norms and then record the
additional cylindrical weights used on \(X_\infty\).

Let \(k\in\mathbb N\) and \(\beta\in\left(0,\frac{1}{2}\right)\).
Let $(r, \phi)$ be the coordinates in Section \ref{subsec:model-spaces-metric} and let $(y_1, \ldots, y_{n-2})$ be local coordinates on $\Sigma$.
Let \(\mathcal T^j\) be the set of degree \(j\) monomials generated by
\(\tilde r\partial_r\), \(\partial_\phi\), and
\(\partial_{y_1},\ldots,\partial_{y_{n-2}}\).  Using a finite cover of \(\Sigma\) and a
subordinate partition of unity, define
\begin{equation}
\label{eq:donaldson-local-norm}
        \|u\|_{\mathcal{D}^{k,\beta}(\mathcal B_\Sigma\cup\mathcal N_{\Sigma,s})}
        :=
        \max_{\substack{0\le j\le k\\ D\in \mathcal T^j}}
        \|Du\|_{\mathcal{C}^{0,\beta}} .
\end{equation}
Up to sign, there is a canonical trivialization of $\mathcal I$ on local coordinate patches. Using this trivialization, we extend the definition in \eqref{eq:donaldson-local-norm} to \(\mathcal I\)-valued sections. Near the branching set, the definition in \eqref{eq:donaldson-local-norm} coincides with the norm given in \cite{Donaldson2021}. On the neck, \eqref{eq:donaldson-local-norm} is equivalent to the ordinary \(\mathcal{C}^{k,\beta}\)-norm for the cylindrical metric.

Next we define the weighted H\"older norm on $\mathcal{A}_\infty$. Let $r_\infty(x):=|x|$ be the Euclidean distance to the origin in $\RR^n$. We can choose \(\rho_\infty: X_s\to[1,\infty)\), independently of $s$, such that
\(\rho_\infty=1\) on \(\mathcal B_\Sigma\cup\mathcal N_{\Sigma,s}\),
\(\rho_\infty=r_\infty\) on \(\mathcal A_\infty\), and \(\rho_\infty\) is nondecreasing
throughout the transition region.

With the function $\rho_\infty$, we set up the weighted H\"older norms on $\mathcal{A}_\infty$. For a section \(u \in \Gamma(\mathcal{I})\) on \(\mathcal A_\infty\), set
\begin{equation}
    \label{eq:euclidean-weighted-D-norm-zeroth-order}
        \|u\|_{\mathcal{D}^{0,\beta}_\gamma(\mathcal A_\infty)}
        :=
        \sup_{\mathcal A_\infty}\rho_\infty^{-\gamma}|u|
        +
        \sup_{\substack{x\ne x'\\ |x-x'|\le \rho_\infty(x)/2}}
        \rho_\infty(x)^{-\gamma+\beta}
        \frac{|u(x)-u(x')|}{|x-x'|^\beta},
\end{equation}
and define
\begin{equation}
    \label{eq:euclidean-weighted-D-norm}
    \|u\|_{\mathcal{D}^{k,\beta}_\delta(\mathcal A_\infty)}
        :=
        \sum_{j=0}^k
        \|\nabla^j u\|_{\mathcal{D}^{0,\beta}_{\delta-j}(\mathcal A_\infty)} .
\end{equation}

Finally, on the transition region set
$
 \|u\|_{\mathcal{D}^{k,\beta}_\delta(\mathcal C_{\Sigma,s})} := \|u\|_{\mathcal{C}^{k,\beta}(\mathcal C_{\Sigma,s})},
$
where $\mathcal{C}^{k,\beta}$ is the standard H\"older norm on $\mathcal C_{\Sigma,s}$.
Using a partition of unity subordinate to the boundary, neck, transition, and asymptotic regions, we combine these local norms to get a global $\mathcal{D}^{k,\beta}_\delta$-norm on $X_s$.
We also use the zeroth-order norm
\begin{equation*}
\begin{aligned}
        \|u\|_{\mathcal{D}^0_\delta(X_s)} := \|\rho_\infty^{-\delta} u\|_{\mathcal{C}^0(X_s)}.
\end{aligned}
\end{equation*}
Equivalent choices of covers and flat local trivializations give uniformly
equivalent norms.

Next we define weighted H\"older norms on $X_\infty$, which we denote as $\|\ldots\|_{\mathcal{D}^{k,\beta}_{\delta, \gamma}(X_\infty)}$. On the transition region and on the asymptotic region, this norm coincides with the $\mathcal{D}^{k,\beta}_{\delta}$ norm, which we defined for the finite neck case. 

On the cylindrical end $\mathcal{N}_{\Sigma_a, \infty}$, consider the coordinates $(t, \phi, y_i)$ which we defined in Section \ref{subsec:model-spaces-metric}.
Choose a smooth nondecreasing function
\(\rho_{\Sigma_a}:\mathcal{N}_{\Sigma_a, \infty}\to(-\infty,0]\) such that
\(\rho_{\Sigma_a}=0\) when $t > -1$ and
\(\rho_{\Sigma_a}=t\) when $t < -2$. Extend $\rho_{\Sigma_a}$ by zero to the whole of $X_\infty$. Define 
\begin{align*}
        \|u\|_{\mathcal{D}^{k, \beta}_{\delta, \gamma}
        (\mathcal{N}_{\Sigma_a} \subset X_\infty)
        } := \|e^{- \gamma \cdot \rho_{\Sigma_a}} u\|_{\mathcal{C}^{k, \beta}(\mathcal{N}_{\Sigma_a})},
\end{align*}
where $\mathcal{C}^{k,\beta}$ is the standard H\"older norm with respect to the cylindrical metric on $\mathcal{N}_{\Sigma_a}$.
Using a partition of unity subordinate to the transition region, the asymptotic region, and the cylindrical ends, we combine these local norms to define $\|\cdot\|_{\mathcal{D}^{k,\beta}_{\delta,\gamma}(X_\infty)}$.

Let \(\Delta_s\) and \(\Delta_\infty\) denote the Laplacians of \(g_s\) and
\(g_\infty\), acting on \(\mathcal I\)-valued sections. 
When the choice of metric is obvious, we write \(\Delta\). For
\(\delta\in\mathbb R\), define
\begin{align*}
        E^{k+2,\beta}_{\delta+2}(X_s)
        :=&
        \left\{
        u\in \mathcal{D}^{k+2,\beta}_{\delta+2}(X_s):
        \Delta_su\in \mathcal{D}^{k,\beta}_{\delta}(X_s)
        \right\}, \text{ and}\\
        E^{k+2,\beta}_{\delta+2,\gamma}(X_\infty)
        :=&
        \mathcal{D}^{k+2,\beta}_{\delta+2,\gamma}(X_\infty).
\end{align*}
We equip $E^{k+2,\beta}_{\delta+2}(X_s)$ with the graph norm
\begin{align*}
        \|u\|_{E^{k+2,\beta}_{\delta+2}(X_s)} 
        := \|u\|_{\mathcal{D}^{k+2,\beta}_{\delta+2}(X_s)} + \|\Delta_su\|_{\mathcal{D}^{k,\beta}_\delta(X_s)},
\end{align*}
while we equip $E^{k+2,\beta}_{\delta+2,\gamma}(X_\infty)$ with the $\mathcal{D}^{k+2,\beta}_{\delta+2,\gamma}$-norm.
In the finite neck case, this is in analogy with \cite[Section 3]{Donaldson2021}. 

\subsection{(Weak) inverses and uniform regularity estimates.}
\label{subsec:uniform-schauder-inverse}
Next we study the invertibility of $\Delta$ using the previous H\"older spaces. To study the cases $X_s$ and $X_\infty$ simultaneously, write $Y$ for either $X_s$ or $X_\infty$. Set \(Y^\circ=X_s\setminus\Sigma\) when $Y = X_s$ and set $Y^\circ = X_\infty$ otherwise.

Let $\Delta_Y:\Gamma(\mathcal I)\to \Gamma(\mathcal I)$ be either $\Delta_s$ or $\Delta_\infty$. Given
\(f\in L^1_{\mathrm{loc}}(Y^\circ;\mathcal I)\), a section
\(u\in H^1_{\mathrm{loc}}(Y^\circ;\mathcal I)\) is a \emph{weak solution}
of \(\Delta_Y u=f\) if
\begin{equation}
\label{eq:weak-solution}
        \int_{Y^\circ}
        \langle \mathrm du,\mathrm d\psi\rangle_{g_Y}\,\mathrm d\mu_Y
        =
        \int_{Y^\circ}
        \langle f,\psi\rangle_{\mathcal I}\,\mathrm d\mu_Y
\end{equation}
for every \(\psi\in C_c^\infty(Y^\circ;\mathcal I)\) and $u$ is called the weak inverse of $f$.

Set \(\delta_*=(2-n)/2\).
Let $0 < \epsilon \ll 1$ and fix $\delta \in (\delta_* - \epsilon, \delta_* + \epsilon)$.
Choose a smooth function
\(\chi:[0,\infty)\to[0,1]\) such that \(\chi=0\) on
\([0,R_\infty]\) and \(\chi=1\) on
\([R_\infty+1,\infty)\). On the Euclidean end, define \(\varrho>0\) by
\begin{equation}
\label{eq:weak-radial-weight}
\varrho^{\delta}
=
-\frac{\delta(n-2+\delta)}{n-2}
\left[
r_\infty^{2-n}
\int_0^{r_\infty}\chi(\tau)\tau^{n+\delta-3}\,\mathrm d\tau
+
\int_{r_\infty}^{\infty}\chi(\tau)\tau^{\delta-1}\,\mathrm d\tau
\right].
\end{equation}

It is constant for \(r_\infty\le R_\infty\), so we extend it constantly
over the rest of \(X_s\) and \(X_\infty\); expressions containing
\(\chi(r_\infty)\) are extended by zero off \(\mathcal A_\infty\).
The function $\varrho$ is selected this way, because it satisfies
$$
\Delta_Y(\varrho^{\delta})
= \chi(r_\infty) \Delta_Y(r_\infty^{\delta}) = - \delta(n-2+\delta) \:\chi(r_\infty)\: r_\infty^{\delta-2} > 0.
$$
Moreover, it is strictly positive and when $r_\infty$ is large, $\varrho^\delta = r_\infty^\delta + \mathcal{O}(r^{2-n}_\infty)$. All these properties can be shown by direct computation.

Let \(H(Y)\) be the completion of
\(C_c^\infty(Y^\circ;\mathcal I)\) with respect to
\begin{equation}
\label{eq:weak-energy-inner-product}
\begin{split}
\langle v,z\rangle_{H(Y)}
&:=
\int_{Y^\circ}
\varrho^{2-n}\langle\mathrm dv,\mathrm dz\rangle_{g_Y}\,\mathrm d\mu_Y\\
&\quad
-\delta(n-2+\delta)
\int_{Y^\circ}
\varrho^{2-n-\delta}\chi(r_\infty)r_\infty^{\delta-2}
\langle v,z\rangle\,\mathrm d\mu_Y.
\end{split}
\end{equation}
On \(X_\infty\),
write \(\ell_\infty:=\sum_a\rho_{\Sigma_a}\), then
\(\ell_\infty=t\) on each cylindrical end.

\begin{lemma}
\label{lemma:weak-invertibility}
There exist \(\epsilon,C>0\) such that the following statements hold for
\(|\delta-\delta_*|<\epsilon\).

For every \(s>0\), if
\(\varrho^{2-\frac n2-\delta}f\in L^2(X_s)\), then
\(\Delta_s u_s=f\) has a unique weak solution satisfying
\(\varrho^{-\delta} u_s\in H(X_s)\). If \(|\gamma|<\epsilon\) and
\(\varrho^{2-\frac n2-\delta}e^{-\gamma\ell_\infty}f
\in L^2(X_\infty)\), then \(\Delta_\infty u_\infty=f\) has a unique
weak solution satisfying
\(\varrho^{-\delta}e^{-\gamma\ell_\infty} u_\infty\in H(X_\infty)\).
Moreover,
\begin{equation}
\label{eq:uniform-weak-inverse-estimate}
\begin{split}
\|\varrho^{-\delta}u_s\|_{H(X_s)}
&\le
C\|\varrho^{2-\frac n2-\delta}f\|_{L^2(X_s)},\\
\|\varrho^{-\delta}e^{-\gamma\ell_\infty}u_\infty\|_{H(X_\infty)}
&\le
C\|\varrho^{2-\frac n2-\delta}e^{-\gamma\ell_\infty}f\|_{L^2(X_\infty)}.
\end{split}
\end{equation}
The constant \(C\) is independent of \(s\), \(\delta\), and \(\gamma\)
in the indicated ranges.
\end{lemma}
\begin{proof}
For the proof, let \((Y,\ell,\eta)=(X_s,0,0)\) or
\((X_\infty,\ell_\infty,\gamma)\). We first prove a uniform
Hardy--Poincar\'e estimate:
\textit{
For every \(v\in H(Y)\),
\begin{equation}
\label{eq:uniform-hardy-poincare}
\|\varrho^{-\frac n2}v\|_{L^2(Y)}
+
\|\varrho^{1-\frac n2}\mathrm dv\|_{L^2(Y)}
\le C\|v\|_{H(Y)},
\end{equation}
where \(C\) is independent of \(s\)}. 

The gradient term is contained in
the definition of \(\|v\|_{H(Y)}\), so it remains to estimate the first
term.

On \(\mathcal A_\infty\cap\{r_\infty\ge R_\infty+1\}\), the
zeroth-order term in \eqref{eq:weak-energy-inner-product} is comparable to
\(\varrho^{-n}|v|^2\). On each finite neck
\(\mathcal N_{\Sigma_a,s}\), and on the corresponding cylindrical end
\(\mathcal N_{\Sigma_a,\infty}\), the monodromy of
\(\mathcal I\) gives the circlewise estimate
\[
\int_{S^1_\phi}|v|^2\,\mathrm d\phi
\le C\int_{S^1_\phi}|\partial_\phi v|^2\,\mathrm d\phi,
\]
because an \(\mathcal I\)-valued section is anti-periodic in \(\phi\).
Integrating over the remaining variables controls all the necks with a
constant independent of their length. The same estimate controls each
boundary region \(\mathcal B_{\Sigma_a}\). Indeed, on this region,
\[
\mathrm d\mu_{g_s}
=
\varepsilon^{n-2}\tilde r\,\mathrm dr\,\mathrm d\phi\,
\mathrm d\mu_{h_{\Sigma_a}},
\qquad
\tilde r^{-2}|\partial_\phi v|^2\,\mathrm d\mu_{g_s}
=
\varepsilon^{n-2}\tilde r^{-1}|\partial_\phi v|^2\,
\mathrm dr\,\mathrm d\phi\,\mathrm d\mu_{h_{\Sigma_a}}.
\]
Since \(0<\tilde r\le1\), the circlewise estimate gives
\[
\int_{\mathcal B_{\Sigma_a}}|v|^2\,\mathrm d\mu_{g_s}
\le
C\int_{\mathcal B_{\Sigma_a}}|\mathrm dv|_{g_s}^2\,
\mathrm d\mu_{g_s}.
\]
Because \(\varrho\) is a non-zero constant on \(\mathcal B_{\Sigma_a}\), this is precisely the
required weighted \(L^2\)-estimate there.

It remains to consider \(\mathcal C_{\Sigma,s}\) (or \(\mathcal C_\Sigma\)
on \(X_\infty\)) and the bounded part of \(\mathcal A_\infty\). These regions can be identified with a fixed compact model and so the standard Poincar\'e inequality has a constant independent of \(s\). Combining the estimates on these
four regions proves \eqref{eq:uniform-hardy-poincare}.

For compactly supported smooth sections \(v,z\), set
\[
b(v,z):=
\int_{Y^\circ}
\langle\mathrm \varrho^{2-n-\delta} e^{-\eta\ell} \Delta (\varrho^\delta e^{\eta\ell} v), z\rangle_{\mathcal I}\,\mathrm d\mu_Y
=
\int_{Y^\circ}
\langle\mathrm d(\varrho^\delta e^{\eta\ell}\: v),\mathrm d(\varrho^{2-n-\delta} e^{-\eta\ell} \: z)\rangle_{g_Y}\,\mathrm d\mu_Y.
\]
Using the properties of $\varrho$ and the fact that the supports of \(\mathrm d\varrho\) and \(\mathrm d\ell\) are disjoint, one can show
\begin{align}
\label{eq:weak-bilinear-expansion}
b(v,z)
={}&\langle v,z\rangle_{H(Y)}
+(2-n-2\delta)
\int_Y\varrho^{2-n}z
\langle\mathrm d\log\varrho,\mathrm dv\rangle\,\mathrm d\mu_Y
\notag\\
&-2\eta
\int_Y\varrho^{2-n}z
\langle\mathrm d\ell,\mathrm dv\rangle\,\mathrm d\mu_Y
+\int_Y\varrho^{2-n}
\bigl[
\eta\Delta_Y\ell-\eta^2|\mathrm d\ell|^2
\bigr]
\langle v,z\rangle_{\mathcal I}\,\mathrm d\mu_Y.
\end{align}
By construction, \(\mathrm d\ell\) and \(\Delta_Y\ell\) are uniformly bounded,
and \(\Delta_Y\ell_\infty\) is supported in fixed cylindrical collars.
Thus, \eqref{eq:uniform-hardy-poincare} and Cauchy--Schwarz imply
\[
|b(v,z)-\langle v,z\rangle_{H(Y)}|
\le
C\bigl(|\delta-\delta_*|+|\eta|+\eta^2\bigr)
\|v\|_{H(Y)}\|z\|_{H(Y)}.
\]
After decreasing \(\epsilon\), the form \(b\) extends to
\(H(Y)\times H(Y)\) and satisfies
\begin{equation}
\label{eq:weak-coercivity}
b(v,v)\ge\frac12\|v\|_{H(Y)}^2.
\end{equation}

Define
\[
F_f(z):=\int_Y \varrho^{2-n-\delta} e^{-\eta\ell} \langle f,z\rangle_{\mathcal I}\,\mathrm d\mu_Y.
\]
By \eqref{eq:uniform-hardy-poincare},
\[
|F_f(z)|
\le
C\|\varrho^{2-\frac n2-\delta}e^{-\eta\ell}f\|_{L^2(Y)}
\|z\|_{H(Y)}.
\]
The Lax--Milgram theorem gives a unique \(v\in H(Y)\) such that
\(b(v,z)=F_f(z)\) for every \(z\in H(Y)\), with
\[
\|v\|_{H(Y)}
\le
C\|\varrho^{2-\frac n2-\delta}e^{-\eta\ell}f\|_{L^2(Y)}.
\]
Set \(u = \varrho^\delta e^{\eta\ell}\: v\). Taking \(z=\varrho^{\frac n2 -2 +\delta}e^{\eta\ell} \: \psi\) for
\(\psi\in C_c^\infty(Y^\circ;\mathcal I)\) gives
\eqref{eq:weak-solution}, and hence \(u\) is a weak solution. Conversely,
every weak solution in the stated energy class satisfies the same
variational identity, so uniqueness follows from
\eqref{eq:weak-coercivity}. The estimate above is
\eqref{eq:uniform-weak-inverse-estimate}.
\end{proof}

\begin{proposition}
\label{prop:uniform-invertibility}
Let $k\in\mathbb N$ and $\beta\in(0,1/2)$.
There exists $\epsilon>0$ such that, for every $\delta,\gamma$ satisfying
$|\delta-\delta_*|<\epsilon$ and $|\gamma|<\epsilon$, and every
$s\ge1$, the operators
\begin{equation}
\begin{aligned}
        \Delta_s&:
        E^{k+2,\beta}_{\delta}(X_s)
        \longrightarrow
        \mathcal{D}^{k,\beta}_{\delta-2}(X_s),\\
        \Delta_\infty&:
        E^{k+2,\beta}_{\delta,\gamma}(X_\infty)
        \longrightarrow
        \mathcal{D}^{k,\beta}_{\delta-2,\gamma}(X_\infty)
\end{aligned}
\end{equation}
are isomorphisms. Moreover, for each such $\delta$ and $\gamma$, there
is a constant $C>0$, independent of $s$, such that
\begin{equation}
\label{eq:uniform-holder-inverse-estimate}
\begin{aligned}
        \|u\|_{E^{k+2,\beta}_{\delta}(X_s)}
        &\le
        C\|\Delta_su\|_{\mathcal{D}^{k,\beta}_{\delta-2}(X_s)},\\
        \|u\|_{E^{k+2,\beta}_{\delta,\gamma}(X_\infty)}
        &\le
        C\|\Delta_\infty u\|_{
        \mathcal{D}^{k,\beta}_{\delta-2,\gamma}(X_\infty)}.
\end{aligned}
\end{equation}
\end{proposition}
\begin{proof}
Let $\epsilon_0$ be the constant in Lemma
\ref{lemma:weak-invertibility}. For each $a$, let $\lambda_a>0$ be the
first eigenvalue of the $\mathcal I$-valued Laplacian on
$S^1_\phi\times\Sigma_a$. The anti-periodicity in the
$S^1_\phi$-direction gives $\lambda_a\ge1/4$. Shrinking $\epsilon$ if
necessary, assume that
\[
        [\delta_*-2\epsilon,\delta_*+2\epsilon]\subset(2-n,0),
        \qquad
        2\epsilon<\epsilon_0,
        \qquad
        \epsilon<\min_a\sqrt{\lambda_a}.
\]

\noindent\textbf{Range:} We first prove surjectivity. Let
$f\in\mathcal{D}^{k,\beta}_{\delta-2}(X_s)$. The Euclidean weighted
H\"older isomorphism for $2-n<\delta<0$, obtained as in
\cite{McOwen1979}, gives, after a bounded extension of
$f|_{\mathcal A_\infty}$ to $\mathbb R^n$, a section $u_{\mathrm E}$
such that $\Delta_{\mathrm{Eucl}}u_{\mathrm E}=f$ sufficiently far out.
If $\chi_{\mathrm E}$ is a fixed cutoff equal to one near infinity, then
$g:=f-\Delta_s(\chi_{\mathrm E}u_{\mathrm E})$ is compactly supported.
Lemma \ref{lemma:weak-invertibility}, applied at weight $\delta$,
gives a weak solution $w$ of $\Delta_sw=g$. Donaldson's local
regularity \cite[Proposition 2.7]{Donaldson2021} and ordinary elliptic
regularity give the required H\"older regularity. On the Euclidean end,
since $\varrho^{-\delta}w\in H(X_s)$, the spherical-harmonic expansion
of $w$ contains no nondecaying modes. Hence,
$w=O(r_\infty^{2-n})$, and
$u=\chi_{\mathrm E}u_{\mathrm E}+w$ belongs to
$E^{k+2,\beta}_\delta(X_s)$ and satisfies $\Delta_su=f$.

Now let
$f\in\mathcal D^{k,\beta}_{\delta-2,\gamma}(X_\infty)$.
Choose $\delta<\widehat\delta<\delta_*+2\epsilon,\;- \epsilon<\widehat\gamma<\gamma. $ The strict inequalities and the prescribed decay of $f$ imply that
\[
\varrho^{2-\frac n2-\widehat\delta}
e^{-\widehat\gamma\ell_\infty}f\in L^2(X_\infty).
\]
Lemma \ref{lemma:weak-invertibility} therefore gives a weak solution
$u$ of $\Delta_\infty u=f$ satisfying
$\varrho^{-\widehat\delta}e^{-\widehat\gamma\ell_\infty}u
\in H(X_\infty)$. Local elliptic regularity applies to $u$ on each end.

Let $u_{\mathrm E}$ be the Euclidean solution constructed above from
$f|_{\mathcal A_\infty}$. The difference $u-u_{\mathrm E}$ is harmonic
sufficiently far out. Since $\delta<\widehat\delta$ and
$\varrho^{-\widehat\delta}e^{-\widehat\gamma\ell_\infty}u
\in H(X_\infty)$, the Euclidean estimate for $u_{\mathrm E}$ gives
\[
\varrho^{-\widehat\delta}e^{-\widehat\gamma\ell_\infty}
\bigl(u-\chi_{\mathrm E}u_{\mathrm E}\bigr)\in H(X_\infty).
\]
The spherical-harmonic expansion of $u-u_{\mathrm E}$ therefore contains
no nondecaying modes. Thus $u-u_{\mathrm E}=O(r_\infty^{2-n})$, so $u$
has Euclidean weight $\delta$.

On the $a$-th cylindrical end, write
$Y_a=S^1_\phi\times\Sigma_a$ and expand $u$ and $f$ in eigenfunctions of
$\Delta_{Y_a}$. If $\lambda\ge\lambda_a$ is an eigenvalue, the
corresponding coefficients satisfy
\[
        -u_\lambda''+\lambda u_\lambda=f_\lambda.
\]
Fix $T<0$ sufficiently negative and set $\kappa:=\sqrt\lambda$. Since
$f_\lambda=O(e^{\gamma t})$ and
$\gamma<\sqrt{\lambda_a}\le\kappa$, for $t\le T$, define
\[
v_\lambda(t)
:=
\frac{1}{2\kappa}
\left(
e^{-\kappa t}\int_{-\infty}^{t}
e^{\kappa\tau}f_\lambda(\tau)\,\mathrm d\tau
+
e^{\kappa t}\int_t^T
e^{-\kappa\tau}f_\lambda(\tau)\,\mathrm d\tau
\right).
\]
Then $-v_\lambda''+\lambda v_\lambda=f_\lambda$ and
$v_\lambda=O(e^{\gamma t})$. The difference is
\[
        u_\lambda-v_\lambda
        =A_\lambda e^{\kappa t}
        +B_\lambda e^{-\kappa t}.
\]
Since $\widehat\gamma<\gamma$, $v_\lambda=O(e^{\gamma t})$, and
$\varrho^{-\widehat\delta}e^{-\widehat\gamma\ell_\infty}u
\in H(X_\infty)$, we have
\[
e^{-\widehat\gamma t}(u_\lambda-v_\lambda)
\in H^1((-\infty,T])
\]
The term
$e^{-\widehat\gamma t}B_\lambda e^{-\kappa t}$ is not square
integrable unless $B_\lambda=0$. Moreover,
$e^{\kappa t}=O(e^{\gamma t})$ as $t\to-\infty$. Hence
$u=O(e^{\gamma t})$ on every cylindrical end. The corresponding
weighted estimates, followed by Schauder estimates on translated unit
cylinders, give
$u\in E^{k+2,\beta}_{\delta,\gamma}(X_\infty)$.

\medskip
\noindent\textbf{Kernel:} For injectivity, choose
$\widehat\delta\in(\delta,\delta_*+2\epsilon)$. By \cite[(2.9)]{Donaldson2021}, $u\sim O(r^{\frac12})$ near $\Sigma$. We conclude that for every harmonic
$u\in E^{k+2,\beta}_\delta(X_s)$,
$\varrho^{-\widehat\delta}u\in H(X_s)$. Weak uniqueness in Lemma
\ref{lemma:weak-invertibility} gives $u=0$. The same argument applies on
$X_\infty$: since $u=O(e^{\gamma t})$,
$\varrho^{-\widehat\delta}u\in H(X_\infty)$ at cylindrical weight $0$,
so weak uniqueness again gives $u=0$.

\medskip
\noindent\textbf{Estimates}:
It remains to prove the uniform estimate on $X_s$. On the Euclidean end,
the weighted Schauder estimate gives
\[
\|u\|_{\mathcal D^{k+2,\beta}_\delta(\mathcal A_\infty)}
\le C\left(
\|\Delta_{\mathrm{Eucl}}u\|_{
\mathcal D^{k,\beta}_{\delta-2}(\mathcal A_\infty)}
+\|\rho_\infty^{-\delta}u\|_{\mathcal C^0(\mathcal A_\infty)}
\right).
\]
Let $Y_a=S^1_\phi\times\Sigma_a$. On the interior of
$\mathcal N_{\Sigma_a,s}=[-s,0]\times Y_a$, we have
$\Delta_s=-\partial_t^2+\Delta_{Y_a}$. Hence, for every $j\in\mathbb Z$
such that $[j-2,j+2]\subset[-s,0]$,
\[
\begin{aligned}
\|u\|_{\mathcal C^{k+2,\beta}([j-1,j+1]\times Y_a)}
\le C\bigl(&
\|\Delta_su\|_{\mathcal C^{k,\beta}([j-2,j+2]\times Y_a)}+\|u\|_{\mathcal C^0([j-2,j+2]\times Y_a)}
\bigr).
\end{aligned}
\]
The constant is independent of $j$ and $s$. Near
$\mathcal B_{\Sigma_a}$ we use
\cite[Proposition 2.2 and (2.4)]{Donaldson2021}, while pullback by
$\Phi_s$ reduces the estimate on $\mathcal C_{\Sigma,s}$ to the ordinary
Schauder estimate on the fixed model $\mathcal C_\Sigma$. Combining these regional estimates gives
\begin{equation}
\label{eq:uniform-schauder-with-zero-order}
        \|u\|_{\mathcal D^{k+2,\beta}_\delta(X_s)}
        \le C\left(
        \|\Delta_su\|_{\mathcal D^{k,\beta}_{\delta-2}(X_s)}
        +\|u\|_{\mathcal D^0_\delta(X_s)}
        \right),
\end{equation}
where $C$ is independent of $s$. We
claim that
\begin{equation}
\label{eq:uniform-zero-order-estimate}
        \|u\|_{\mathcal D^0_\delta(X_s)}
        \le C
        \|\Delta_su\|_{\mathcal D^{0,\beta}_{\delta-2}(X_s)}
\end{equation}
with the same uniformity.

Otherwise, there are $s_i\ge1$ and $u_i$ such that
\[
        \|u_i\|_{\mathcal D^0_\delta(X_{s_i})}=1,
        \qquad
        \|\Delta_{s_i}u_i\|_{
        \mathcal D^{0,\beta}_{\delta-2}(X_{s_i})}\longrightarrow0.
\]
We may assume $s_i\to\infty$. Choose $x_i\in X_{s_i}$ such that
\[
\rho_\infty(x_i)^{-\delta}|u_i(x_i)|\ge\frac12.
\]
Fix $0<\beta'<\beta$ and pass to a subsequence whenever necessary.

\medskip
\noindent\textit{Case 1:}
$x_i\in\mathcal A_\infty$ and
$R_i:=r_\infty(x_i)\to\infty$.
Set
\[
\widetilde u_i(y):=R_i^{-\delta}u_i(R_i y).
\]
After passing to a subsequence, $x_i/R_i\to y_\infty\in S^{n-1}$ and
\[
\widetilde u_i\longrightarrow\widetilde u
\quad\text{in}\quad
\mathcal C^{2,\beta'}_{\mathrm{loc}}(\RR^n\setminus\{0\}),
\qquad
|\widetilde u(y_\infty)|\ge\frac12.
\]
Moreover, $|\widetilde u(y)|\le |y|^\delta$. If $Y_j$ is a spherical
harmonic of degree $j$, neither $r^jY_j$ nor $r^{2-n-j}Y_j$ satisfies
this bound at both $0$ and infinity when $2-n<\delta<0$. Thus
$\widetilde u=0$, a contradiction.

After passing to a subsequence, fix $a$ whenever $x_i$ lies in a
boundary or neck region. When
$x_i\in\mathcal N_{\Sigma_a,s_i}=[-s_i,0]\times Y_a$, write
$x_i=(t_i,z_i)$.

\medskip
\noindent\textit{Case 2:}
For some $L>0$,
\[
x_i\in
\bigl(\mathcal A_\infty\cap B_L(0)\bigr)
\cup\mathcal C_{\Sigma,s_i}
\cup\bigl([-L,0]\times Y_a\bigr).
\]
Then
\[
u_i\longrightarrow u_\infty
\quad\text{in}\quad
\mathcal C^{2,\beta'}_{\mathrm{loc}}(X_\infty),
\]
where $u_\infty$ is nonzero, bounded, and harmonic. The estimate
$|u_i|\le\rho_\infty^\delta$ passes to the limit. The
spherical-harmonic expansion on the Euclidean end and the
$\Delta_{Y_a}$-eigenfunction expansion on each cylindrical end give,
for some $\sigma>0$,
\[
u_\infty=O(r_\infty^{2-n})
\quad\text{on }\mathcal A_\infty,
\qquad
u_\infty=O(e^{\sigma t})
\quad\text{on }\mathcal N_{\Sigma_a,\infty}.
\]
These decay estimates imply
$\varrho^{-\delta_*}u_\infty\in H(X_\infty)$. Applying Lemma
\ref{lemma:weak-invertibility} with $f=0$ and
$(\delta,\gamma)=(\delta_*,0)$ gives $u_\infty=0$, contradicting the
choice of $u_\infty$.

\medskip
\noindent\textit{Case 3:}
For some $L>0$,
\[
x_i\in
\mathcal B_{\Sigma_a}
\cup\bigl([-s_i,-s_i+L]\times Y_a\bigr).
\]
Set $\tau=t+s_i$ on $\mathcal N_{\Sigma_a,s_i}$. Then $u_i$ converges
to a nonzero bounded harmonic section $u_{\mathcal B}$ with
\[
u_i\longrightarrow u_{\mathcal B}
\quad\text{in}\quad
\mathcal D^{2,\beta'}(\mathcal B_{\Sigma_a})
\quad\text{and}\quad
\mathcal C^{2,\beta'}_{\mathrm{loc}}([0,\infty)\times Y_a).
\]
The spectral gap gives $u_{\mathcal B}=O(e^{-\sigma\tau})$ for some
$\sigma>0$. Since $u_{\mathcal B}=O(r^{1/2})$ at $\Sigma_a$, integration
by parts gives $\int_{\mathcal B_{\Sigma_a}\cup([0,\infty)\times Y_a)}
|\mathrm du_{\mathcal B}|^2=0.$ As $\mathcal{I}$ has nontrivial monodromy, we conclude that $u_{\mathcal{B}}=0$, a contradiction.

\medskip
\noindent\textit{Case 4:}
\[
x_i=(t_i,z_i)\in\mathcal N_{\Sigma_a,s_i},
\qquad
-t_i\longrightarrow\infty,
\qquad
t_i+s_i\longrightarrow\infty.
\]
Here $-t_i$ and $t_i+s_i$ are the distances from $x_i$ to the two
ends of the neck. Define
\[
\widehat u_i(\tau,z):=u_i(t_i+\tau,z),
\qquad
\tau\in[-(t_i+s_i),-t_i].
\]
Then $\widehat u_i$ converges in
$\mathcal C^{2,\beta'}_{\mathrm{loc}}(\mathbb R\times Y_a)$ to a
nonzero bounded harmonic section $u_{\mathrm{cyl}}$. Since every
eigenvalue $\lambda$ of $\Delta_{Y_a}$ is positive,
\[
u_{\mathrm{cyl}}(\tau,z)
=\sum_\lambda
\left(A_\lambda e^{\sqrt\lambda\tau}
+B_\lambda e^{-\sqrt\lambda\tau}\right)\phi_\lambda(z).
\]
Boundedness gives $A_\lambda=B_\lambda=0$ for every $\lambda$, a
contradiction. This proves \eqref{eq:uniform-zero-order-estimate};
\eqref{eq:uniform-holder-inverse-estimate} now follows from
\eqref{eq:uniform-schauder-with-zero-order}.
\end{proof}

\section{Asymptotic expansions for polynomially growing harmonic functions}
\label{sec:asymptotic-expansions}
In this section, we study harmonic sections of $\mathcal{I}$ with prescribed polynomial asymptotics on the Euclidean end. We repeat the analysis from \cite{salm2024construction} and study the Fourier modes on the necks and their dependence on the parameter $s$. Like in \cite{salm2024construction}, we also relate these Fourier modes to the asymptotic expansion given in \cite{Donaldson2021}.

\subsection{Harmonic sections with polynomial growth.}
\label{subsec:polynomial-asymptotics}
In this subsection, we study harmonic sections of $\mathcal{I}$ with polynomial growth.
Let
$$
    \mathcal H_{\mathrm{poly}}(\mathbb R^n)
    :=
    \{P\in\mathbb R[x_1,\ldots,x_n]:
    \Delta_{\mathrm{Eucl}}P=0\}
$$
denote the space of harmonic polynomials on $\RR^n$.
We claim that for every $P \in \mathcal{H}_{\mathrm{poly}}$, there exists a harmonic section $u_{P,s} \in \Gamma(X_s,\mathcal{I})$ such that $P-u_{P,s}$ decays on $\mathcal{A}_\infty$. Indeed, let $\chi_\infty$ be a smooth step function supported on $\mathcal{A}_\infty$ that equals $1$ when $r_\infty \gg 0$. Notice that $\Delta_s (\chi_\infty P)$ is compactly supported. Therefore, there exists a $v \in E^{k+2, \beta}_{\delta}(X_s)$ such that $\Delta_s v = - \Delta(\chi_\infty P)$. Setting
$u_{P,s} :=v + \chi_\infty P$ and repeating the same argument on $X_\infty$ gives the claimed harmonic sections. This existence result is due to \cite{Sun2022Z2Pell} and \cite{Newexamples}.
\begin{proposition}
    \label{prop:existence-harmonic-sections-with-polynomial-growth}
    Let $k\in\mathbb N$, $\beta\in(0,\frac{1}{2})$, $\delta = \frac{2-n}{2}$, and $0<\gamma\ll 1$.  For every $P\in\mathcal H_{\mathrm{poly}}(\mathbb R^n)$, the following hold.
\begin{enumerate}
\item[{\rm (i)}]
For every finite neck length $s$, there is a unique $\mathcal I$-valued
harmonic section $u_{P,s}$ on $X_s$ such that $
        \chi_\infty P-u_{P,s}
        \in E^{k+2,\beta}_\delta(X_s;\mathcal I)$.
        The assignment $P\mapsto u_{P,s}$ is real linear.

\item[{\rm (ii)}] Similarly, there is a unique $\mathcal I$-valued harmonic section $u_{P,\infty}$ on $X_\infty$ such that $\chi_\infty P-u_{P,\infty}
        \in E^{k+2,\beta}_{\delta,\gamma}(X_\infty;\mathcal I).$
The assignment $P\mapsto u_{P,\infty}$ is real linear.
\end{enumerate}
\end{proposition}

Given $u_{P,s}$ and $u_{P,\infty}$, we will consider their Fourier modes and fix the coefficient conventions used throughout this paper:
First, rescale the metrics on the components of $\Sigma$, so that $\operatorname{Vol}_{h_\Sigma}(\Sigma_a)=1$, for $a=1,\ldots,p$.
Let $\{\psi_{a,j}\}_{j\ge0}$ be an \(L^2\)-orthonormal eigenbasis of
$\Delta_{h_{\Sigma_a}}$, with
\[
        \Delta_{h_{\Sigma_a}}\psi_{a,j}
        =
        \mu_{a,j}\psi_{a,j},
        \qquad
        0=\mu_{a,0}<\mu_{a,1}\le\mu_{a,2}\le\cdots,
\]
with normalization $\psi_{a,0}=1$. We also set $\mu_1:=\min_{1\le a\le p}\mu_{a,1}$. 
On the neck and boundary regions  $\mathcal B_{\Sigma_a} \cup \mathcal N_{\Sigma_a, s}$, $u_{P,s}$ can uniquely be written as
$$
u_{P,s} = \sum_{q \in \ZZ + \frac{1}{2},\: j \ge 0} u_{q,a,j}(r) e^{i q \phi} \psi_{a,j}
$$
and $u_{q,a,j}(r)$ must satisfy the equation
\begin{equation}
\label{eq:radial-ode}
        \frac1{\tilde r(r)}
        \frac{d}{dr}
        \left(
        \tilde r(r)\frac{d u_{q,a,j}}{dr}
        \right)
        -
        \left(
        \frac{q^2}{\tilde r(r)^2}
        +
        \varepsilon^{-2}\mu_{a,j}
        \right)u_{q,a,j}
        =
        0.
\end{equation}
Let $I_{q,a,j}(r)$ be the unique solution to \eqref{eq:radial-ode}
that satisfies \(I_{q,a,j}(r)=r^{|q|}+O(r^{|q|+1})\) near $r = 0$. On the region where $\tilde r=r$, it agrees up to normalization with the regular solution of the modified Bessel equation. With this normalization we define the Fourier coefficients of $u_{P,s}$ as follows:
\begin{definition}
        \label{def:fourier-coefficients}
    The Fourier coefficients $c_{q, a, j}(P,s)$ are defined by the relationship
    $$
    u_{P,s} = \sum_{q \in \ZZ + \frac{1}{2},\: j \ge 0} c_{q,a,j}(P,s) \: I_{q, a,j}(r)\: e^{i q \phi} \:\psi_{a,j}
    $$
    on $\mathcal B_{\Sigma_a} \cup \mathcal N_{\Sigma_a, s}$.
\end{definition}

Since our setup is similar to that in \cite{salm2024construction}, we directly get a decay estimate for $c_{q,a,j}(P,s)$:
\begin{proposition}[{\cite[Lemma~2.5]{salm2024construction}}]
\label{prop:mode-estimate}
For every $\beta\in(0,\frac12)$, there is a constant
$C>0$, independent of $s,q,a,j$, such that for every $P\in\mathcal H_{\mathrm{poly}}(\mathbb R^n)$ and $s\ge1$
\begin{equation}
\label{eq:mode-estimate}
        |c_{q,a,j}(P,s)|
        \le
        \frac{C}{I_{q,a,j}(R_0)}
        e^{-s \sqrt{q^2 + \frac{\mu_{a,j}}{\varepsilon^2}}} \left\| P \right\|_{\mathcal{D}^{1,\beta}(\operatorname{Supp}(\d \chi_\infty))}.
\end{equation}
\end{proposition}
\begin{proof}
        Set $v_{P,s}:=u_{P,s}-\chi_\infty P$. By Proposition \ref{prop:existence-harmonic-sections-with-polynomial-growth},
        $v_{P,s}\in E^{2,\beta}_\delta(X_s)$ and
        $\Delta_s v_{P,s}=-\Delta_s(\chi_\infty P)$. Since $\chi_\infty$ vanishes on $\mathcal N_{\Sigma_a,s}$, Proposition \ref{prop:uniform-invertibility} gives
        \begin{align}
                \label{eq:estimate_we-need-later}
        \|u_{P,s}\|_{\mathcal C^0(\mathcal N_{\Sigma_a,s})}=\|v_{P,s}\|_{\mathcal C^0(\mathcal N_{\Sigma_a,s})}
        \le& C\|v_{P,s}\|_{E^{2,\beta}_\delta(X_s)} \\
        \le& C\|\Delta_s(\chi_\infty P)\|_{\mathcal D^{0,\beta}_{\delta-2}(X_s)}
        \le C\|P\|_{\mathcal D^{1,\beta}(\operatorname{Supp}(\d\chi_\infty))}.
        \notag
        \end{align}
        
        The map $
        \pi_{q, a, j}(u) := \int_{\{ r \} \times S^1 \times \Sigma_a} 
        e^{- i q \phi} 
        \bar{\psi}_{a,j} 
        \cdot u \:
        \d \phi \wedge \operatorname{vol}_{\Sigma_a}
        $
        is a bounded map on $\mathcal{C}^0(\mathcal{N}_{\Sigma_a,s})$, with operator norm at most $2 \pi$ and it satisfies $\pi_{q, a, j} (u_{P,s}) 
        =
        2 \pi c_{q,a, j}(P,s) \: I_{q, a,j}(r).$
        Because $I_{q,a,j}$ is a strictly increasing function,
        $$
        |c_{q,a,j}(P,s)| \le \frac{1}{I_{q,a,j}(R_0 + s)} \|u_{P,s}\|_{\mathcal{C}^0(\mathcal{N}_{\Sigma_a,s})}.
        $$
        Hence, we only need to estimate the ratio $\frac{I_{q,a,j}(R_0)}{I_{q,a,j}(R_0 + s)}$.
        On the neck region, we can write
        \begin{equation}
                \label{eq:estimate_bessel_function}
                I_{q,a,j}(r) = I_{q,a,j}(R_0)\frac{e^{\kappa (r-R_0)} + b_{q,a,j} \cdot e^{-\kappa (r-R_0)} }{1 + b_{q,a,j}}
        \end{equation}
        for some $b_{q,a,j} \in \RR$ and $\kappa = \sqrt{q^2 + \frac{\mu_{a,j}}{\varepsilon^2}}$. Because $I_{q,a,j}$ is a positive, strictly increasing function, $I_{q, a,j} (r) > 0$ and $I'_{q, a,j} (r) > 0$, which implies $b_{q,a,j} \in (-1,1)$. Therefore, 
        $$
        \frac{I_{q,a,j}(R_0)}{I_{q,a,j}(R_0+s)} 
        = e^{- \kappa s} \frac{1 + b_{q,a,j}}{1 + b_{q,a,j} \cdot e^{-2\kappa s}}.
        $$
        For $\kappa \: s > 0$, the function $f(b) = \frac{1 + b}{1 + b e^{-2\kappa s}}$ is non-decreasing in $b$ and so on the interval $b \in [-1,1]$, its maximum is less than 2. Therefore,
        $$
        \frac{I_{q,a,j}(R_0)}{I_{q,a,j}(R_0+s)} 
        =  e^{- \kappa s} \frac{1 + b_{q,a,j}}{1 + b_{q,a,j} \cdot e^{-2\kappa s}} \le 2 e^{- \kappa s}.
        $$
\end{proof}

We now compare the Fourier coefficients of $u_{P,s}$ and $u_{P,\infty}$.
Let $\mathcal{N}_{\Sigma_a, \infty}:=(-\infty, 0]_t \times S^1_\phi\times\Sigma_a$ be a cylindrical end and let $t\in(-\infty,0]$ be the longitudinal coordinate on it. Under the identification of $t=r-(R_0+s)$, we view $\mathcal N_{\Sigma_a,s}$ as the subset $[-s,0]\times S^1_\phi\times\Sigma_a \subset \mathcal N_{\Sigma_a,\infty}$.
Using the \(L^2\)-orthonormal eigenbasis $\{\psi_{a,j}\}_{j\ge0}$ as before, write
$$
u_{P,\infty} = \sum_{q \in \ZZ + \frac{1}{2},\: j \ge 0} u^\infty_{q,a,j}(t) e^{i q \phi} \psi_{a,j},
$$
where $u^\infty_{q,a,j}(t)$ satisfies the equation
\begin{equation}
\label{eq:radial-ode-neck}
        \frac{\partial^2 u^\infty_{q,a,j}}{\partial t^2}
        -
        \left(
        q^2
        +
        \varepsilon^{-2}\mu_{a,j}
        \right)u^\infty_{q,a,j}
        =
        0.
\end{equation}
As $u_{P,\infty}=\mathcal{O}(e^{\gamma t})$ on the cylindrical end, $u^\infty_{q,a,j}$ is a linear multiple of $e^{t \sqrt{q^2 + \frac{\mu_{a,j}}{\varepsilon^2}}}$. We therefore define the following coefficients.

\begin{definition}
        \label{def:expansion_of_u_P_infty}
    The coefficients $c^\infty_{q, a, j}(P)$ are defined by the relationship
    $$
    u_{P,\infty} = \sum_{q \in \ZZ + \frac{1}{2},\: j \ge 0} c^\infty_{q,a,j}(P) \: e^{t \sqrt{q^2 + \frac{\mu_{a,j}}{\varepsilon^2}}}
    \: e^{i q \phi} \:\psi_{a,j}
    $$
    on $\mathcal N_{\Sigma_a, \infty}$.
\end{definition}
Using an argument similar to that of Proposition \ref{prop:mode-estimate}, one can compare $c^\infty_{q,a,j}(P)$ with $c_{q,a,j}(P,s)$. 
In the next proposition, we work out the case where $j = 0$, because $I_{q,a,0}(r)$ can be written down explicitly.
\begin{proposition}
        \label{prop:limiting-behaviour-fourier-coefficients}
        For every $\beta\in(0,\frac12)$, there are constants
        $C>0$ and $0<\gamma\ll1$, independent of $s,q,a$, such that for every $s\ge1$, $q\in\ZZ+\frac12$, $a\in\{1,\ldots,p\}$, and $P\in\mathcal H_{\mathrm{poly}}(\mathbb R^n)$,
\begin{equation}
        |c_{q,a,0}(P,s) I_{q, a, 0}(R_0) e^{|q| s} - c^\infty_{q,a,0}(P)|
        \le
        C e^{-\gamma s} \left\| P \right\|_{\mathcal{D}^{1,\beta}(\operatorname{Supp}(\d \chi_\infty))}.        
\end{equation}
\end{proposition}
\begin{proof}
        Let $u_{P,s}$ and $u_{P,\infty}$ be as in Proposition \ref{prop:existence-harmonic-sections-with-polynomial-growth}, and set
        $v_\infty:=u_{P,\infty}-\chi_\infty P$.
        Let $\chi_s\colon X_s\to[0,1]$ be a smooth cutoff such that $\chi_s=0$ for $r<R_0$ and $\chi_s=1$ for $r>R_0+1$. Under the identifications in Section \ref{subsec:model-spaces-metric}, $\chi_su_{P,\infty}$ extends by zero to an $\mathcal I$-valued section on $X_s$. 
        Define $w_s:=u_{P,s}-\chi_su_{P,\infty}$.
        
        Since $u_{P,s}$ and $u_{P,\infty}$ have the same prescribed polynomial asymptotics and $\chi_s=1$ on the Euclidean end, it follows that $w_s\in E^{2,\beta}_\delta(X_s)$.

        The support of $\d\chi_s$ lies in the region where $t\in[-s,-s+1]$. Here $u_{P,\infty}=v_\infty$. Proposition \ref{prop:uniform-invertibility} and the cylindrical weight therefore give
        \begin{align*}
        \|w_s\|_{E^{2,\beta}_\delta(X_s)}
        \le C\|\Delta_s(\chi_su_{P,\infty})\|_{\mathcal D^{0,\beta}_{\delta-2}(X_s)}\le Ce^{-\gamma s}\|v_\infty\|_{E^{2,\beta}_{\delta,\gamma}(X_\infty)}\le Ce^{-\gamma s}\|P\|_{\mathcal D^{1,\beta}(\operatorname{Supp}(\d\chi_\infty))}.
        \end{align*}

        The function
        $\exp\left(|q|\int_{R_0}^r\frac{\d\tau}{\tilde r(\tau)}\right)$
        solves \eqref{eq:radial-ode} in the case when $j = 0$ and hence
        \[
        I_{q,a,0}(r)=I_{q,a,0}(R_0)
        \exp\left(|q|\int_{R_0}^r\frac{\d\tau}{\tilde r(\tau)}\right).
        \]
        This implies
        $$
        I_{q,a,0}(R_0+s)=I_{q,a,0}(R_0)e^{|q|s}.
        $$
        The fibrewise map
        \[
        \pi_{q,a,0}(u)(r):=
        \int_{\{r\}\times S^1\times\Sigma_a}
        e^{-iq\phi}u\,\d\phi\wedge\operatorname{vol}_{\Sigma_a}
        \]
        is bounded on $\mathcal C^0(\mathcal N_{\Sigma_a,s})$ with operator norm at most $2\pi$. At the terminal slice $r=R_0+s$, or equivalently $t=0$, we have $u_{P,s}-u_{P,\infty}=w_s$ and
        \[
        \pi_{q,a,0}(u_{P,s}-u_{P,\infty})(R_0+s)
        =2\pi\left(
        c_{q,a,0}(P,s)I_{q,a,0}(R_0)e^{|q|s}
        -c^\infty_{q,a,0}(P)
        \right).
        \]
        The preceding estimate for $w_s$ proves the proposition.
\end{proof}

\subsection{Donaldson's expansions near the branching set}
\label{sec:donaldson-expansion}
In \cite{Donaldson2021}, Donaldson gives a description of the asymptotic behaviour of $L^2$-bounded harmonic sections of $\mathcal{I}$. Namely, when the normal bundle $\mathcal{N}$ of the branching set $\Sigma$ is oriented, one can view the normal bundle as a complex line bundle. Locally, near each $y \in \Sigma$, there is a local coordinate $z$ on $\mathcal{N}$, such that $|z|$ is the geodesic distance from the branching set. Donaldson showed that near $y \in \Sigma$, any $L^2$-bounded section $u$ of $\mathcal{I}$ that is harmonic near the branching set can be written as
\begin{equation}
        \label{eq:Expansion-Donaldson-general}
        u = \left(\operatorname{Re} \left( A(y) z^{\frac{1}{2}} + B(y) z^{\frac{3}{2}}\right)
- \frac{1}{2}\operatorname{Re} \left( A(y) z^{\frac{1}{2}}\right)
\operatorname{Re} \left( \bar{\mu}(y) z\right)
\right) + \mathcal{O}\left(|z|^{\frac{5}{2}}\right),
\end{equation}
where $A, B$ are local functions from $\Sigma$ to $\CC$ and $\mu$ is the mean curvature.
Actually, Donaldson showed that $A$ and $B$ can be viewed as sections of a fractional power of $\mathcal{N}$. Explicitly, $A \in \Gamma(\mathcal{N}^{-\frac{1}{2}})$ and $B \in \Gamma(\mathcal{N}^{-\frac{3}{2}})$. The regularity of $A$ and $B$ depends on the regularity of $u$.
More precisely, for \(u\in E^{k+2,\beta}_{\delta+2}(X_s)\), one has \(A(u)\in \mathcal{C}^{k+1,\beta+1/2}(\Sigma;\CC)\) and \(B(u)\in \mathcal{C}^{k,\beta+1/2}(\Sigma;\CC)\). 

In our case, a neighbourhood of each component $\Sigma_a$ is modeled on $\RR^2 \times \Sigma_a$ with the product metric. Therefore, the mean curvature is zero and we get the expansion.
\begin{equation}
        \label{eq:Expansion-Donaldson}
        u_{P,s} = \operatorname{Re} \left( A_{P,s}(y) z^{\frac{1}{2}} + B_{P,s}(y) z^{\frac{3}{2}}\right) + \mathcal{O}\left(|z|^{\frac{5}{2}}\right),
\end{equation}
We use the notation $A_{P,s}$ and $B_{P,s}$ from \eqref{eq:Expansion-Donaldson} throughout the document. 
Denote the restrictions of $A_{P,s}$ and $B_{P,s}$ to $\Sigma_a$ by $A_{P,s,a}$ and $B_{P,s,a}$. Also denote
\begin{align}\label{eq:Averages_on_X}
        A^{\mathrm{av}}_{P,s,a}
        :=
        \int_{\Sigma_a}A_{P,s,a}(y)\,\d\operatorname{vol}_{h_\Sigma}, \qquad
        B^{\mathrm{av}}_{P,s,a}
        :=
        \int_{\Sigma_a}B_{P,s,a}(y)\,\d\operatorname{vol}_{h_\Sigma}.
\end{align}
Using \(I_{q,a,j}(r)=r^{|q|}+O(r^{|q|+1})\) and $z=re^{i\phi}$, we compare \eqref{eq:Expansion-Donaldson} to the expansion in Definition \ref{def:fourier-coefficients} and conclude
\begin{align}\label{eq:A,B_terms=sum_separate_modes}
        A_{P,s}|_{\Sigma_a}
        =
        \sum_{j\ge0}c_{\frac12,a,j}(P,s)\psi_{a,j},
        \qquad
        B_{P,s}|_{\Sigma_a}
        =
        \sum_{j\ge0}c_{\frac32,a,j}(P,s)\psi_{a,j}.
\end{align}
Moreover, $A_{P,s,a}^{\mathrm{av}}=c_{\frac12,a,0}(P,s)$ and $B_{P,s,a}^{\mathrm{av}}=c_{\frac32,a,0}(P,s)$.
As we already know the decay rate of $c_{q,a,0}(P,s)$, we get an estimate on the decay rate of $A_{P,s}$ and $B_{{P,s}}$:
\begin{proposition}
        \label{prop:decay-rate_A_and_B}
        Let $\mu_1:=\min_{1\le a\le p}\mu_{a,1}$.
        For every $k \in \NN$ and $\beta \in (0, \frac{1}{2})$, there is a constant
        $C>0$, independent of $s$, such that for every $P\in\mathcal H_{\mathrm{poly}}(\mathbb R^n)$ and $s > 1$,
        \begin{align*}
        \left\|A_{P,s} - A_{P,s}^{\mathrm{av}}\right\|_{\mathcal{C}^{k+1,\beta+1/2}(\Sigma)} &\le C \:
                e^{-s \sqrt{\frac{1}{4} + \frac{\mu_1}{\varepsilon^2}}} \left\| P \right\|_{\mathcal{D}^{1,\beta}(\operatorname{Supp}(\d \chi_\infty))}, \\
        \left\|B_{P,s} - B_{P,s}^{\mathrm{av}}\right\|_{\mathcal{C}^{k,\beta+1/2}(\Sigma)} &\le C \:
                e^{-s \sqrt{\frac{9}{4} + \frac{\mu_1}{\varepsilon^2}}}\left\| P \right\|_{\mathcal{D}^{1,\beta}(\operatorname{Supp}(\d \chi_\infty))}.
        \end{align*}
\end{proposition}
\begin{proof}
        The argument is the same as that of \cite[Lemma 3.2]{salm2024construction} up to some minor tweaks. Fix $a$ and let $\chi$ be a cutoff such that $\chi(r)=1$ for $r<R_0+\frac13$ and $\chi(r)=0$ for $r>R_0+\frac23$. Set
        \[
        \widetilde u_A
        :=
        \operatorname{Re}\left(
        \sum_{j\ge1}c_{\frac12,a,j}(P,s)
        I_{\frac12,a,j}(r)e^{\frac12 i\phi}\psi_{a,j}
        \right).
        \]
        Let $\Delta_1$ be the Laplacian on $X_s$ with $s = 1$.
        By \eqref{eq:A,B_terms=sum_separate_modes}, $A_{P,s,a}-A_{P,s,a}^{\mathrm{av}}$ is the $A$-coefficient of $\chi\widetilde u_A$. Donaldson's coefficient estimate in \cite{Donaldson2021}, Proposition \ref{prop:uniform-invertibility}, and interior elliptic regularity give
        \begin{align*}
        \|A_{P,s,a}-A_{P,s,a}^{\mathrm{av}}\|_{\mathcal C^{k+1,\beta+1/2}(\Sigma_a)}
        &\le C\|\chi\widetilde u_A\|_{E^{k+2,\beta}_\delta(X_1)}\\
        &\le C\|\Delta_1(\chi\widetilde u_A)\|_{\mathcal D^{k,\beta}_{\delta-2}(X_1)}\\
        &\le C\|\widetilde u_A\|_{\mathcal C^{k+1,\beta}(\operatorname{Supp}(\d\chi))}\\
        &\le C\|\widetilde u_A\|_{L^2([R_0,R_0+1]\times S^1\times\Sigma_a)}.
        \end{align*}

        Put $\kappa_{a,j}:=\sqrt{\frac14+\frac{\mu_{a,j}}{\varepsilon^2}}$. For $r\in[R_0,R_0+1]$, the estimate in \eqref{eq:estimate_bessel_function} gives
        \[
        \frac{I_{\frac12,a,j}(r)}{I_{\frac12,a,j}(r+s-1)}
        \le 2e^{(1-s)\kappa_{a,j}}.
        \]
        Hence, by the orthogonality of the modes,
        \begin{align*}
        \|\widetilde u_A\|_{L^2([R_0,R_0+1]\times S^1\times\Sigma_a)}^2
        &=\pi\sum_{j\ge1}|c_{\frac12,a,j}(P,s)|^2
        \int_{R_0}^{R_0+1}I_{\frac12,a,j}(r)^2\,\d r\\
        &\le C e^{-2(s-1)\sqrt{\frac14+\frac{\mu_1}{\varepsilon^2}}}
        \|u_{P,s}\|_{L^2([R_0+s-1,R_0+s]\times S^1\times\Sigma_a)}^2.
        \end{align*}
        Taking square roots and using \eqref{eq:estimate_we-need-later}, we obtain
        \[
        \|A_{P,s,a}-A_{P,s,a}^{\mathrm{av}}\|_{\mathcal C^{k+1,\beta+1/2}(\Sigma_a)}
        \le
        C e^{-s\sqrt{\frac14+\frac{\mu_1}{\varepsilon^2}}}
        \|P\|_{\mathcal D^{1,\beta}(\operatorname{Supp}(\d\chi_\infty))}.
        \]
        The same argument, with $q=\frac32$, proves the estimate for $B$.
\end{proof}

For the Nash--Moser theorem, it is important that $A^{\mathrm{av}}_{P,s,a}=c_{\frac12,a,0}(P,s)$ and $B^{\mathrm{av}}_{P,s,a}=c_{\frac32,a,0}(P,s)$ are the dominant modes in the expansion of $u_{P,s}$. We know by Proposition \ref{prop:mode-estimate} that $c_{q,a,j}(P,s)=\mathcal{O}(e^{-s \sqrt{q^2 + \frac{\mu_{a,j}}{\varepsilon^2}}})$. Hence, for $A^{\mathrm{av}}_{P,s,a}$ and $B^{\mathrm{av}}_{P,s,a}$ to dominate, we need that 
$$
\frac{3}{2} < \sqrt{\left(\frac{1}{2}\right)^2 + \frac{\mu_{a,j}}{\varepsilon^2}}
$$
for all $a \in \{1, \ldots, p \}$ and $j\ge1$.
This happens when $2 \varepsilon^2 < \mu_1$ and so we define:
\begin{definition}
        \label{def:epsilon-in-metric}
Let $\mu_1:=\min_{1\le a\le p}\mu_{a,1}$ and recall the metric on the neck and boundary regions
\begin{align*}
        g_s=&\mathrm d r^2+\mathrm d\phi^2+\varepsilon^2h_{\Sigma_a}, \\
        g_\infty=&\mathrm d t^2+\mathrm d\phi^2+\varepsilon^2h_{\Sigma_a}.
\end{align*}
In these definitions, set $\varepsilon := \frac{\sqrt{\mu_1}}{2}$.
\end{definition}

By Proposition \ref{prop:limiting-behaviour-fourier-coefficients} we also know the decay behaviour of $A_{P,s}^{\mathrm{av}}$ and $B_{P,s}^{\mathrm{av}}$. With this in mind, define
\begin{align}\label{eq:Averages_on_X_infinity}
        A^{\mathrm{av}}_{P,\infty,a}
        :=
        \frac{c^\infty_{\frac12,a,0}(P)}{I_{\frac{1}{2}, a, 0}( R_0)}, \qquad
        B^{\mathrm{av}}_{P,\infty,a}
        :=
        \frac{c^\infty_{\frac32,a,0}(P)}{I_{\frac{3}{2}, a, 0}( R_0)}.
\end{align}
\begin{remark}
        Notice that $A^{\mathrm{av}}_{P,\infty,a}$ can be defined even if $P$ is not a harmonic polynomial. 
        Indeed, to define $A^{\mathrm{av}}_{P, \infty,a}$, we only needed the asymptotic expansion of Definition \ref{def:expansion_of_u_P_infty}. 
        Such an expansion exists for any $u \in \Gamma(\mathcal{I})$ such that $\Delta(u)$ is supported away from the cylindrical ends. 
        For such $u$, we write $A^{\mathrm{av}}_{\infty,a}(u)$. By definition, $A^{\mathrm{av}}_{P,\infty,a} = A^{\mathrm{av}}_{\infty,a}(u_{P,\infty})$.
\end{remark}

Combining Proposition \ref{prop:limiting-behaviour-fourier-coefficients} with  Proposition \ref{prop:decay-rate_A_and_B}, we conclude 
\begin{proposition}
        \label{prop:decay-rate_A_and_B-in-limit}
        Let $\mu_1:=\min_{1\le a\le p}\mu_{a,1}$.
        For every $k \in \NN$ and $\beta \in (0, \frac{1}{2})$, there are constants
        $C>0$ and $0 < \gamma \ll 1$, independent of $s$, such that for every $P\in\mathcal H_{\mathrm{poly}}(\mathbb R^n)$, $a \in \{1, \ldots, p \}$ and $s > 0$,
        \begin{align*}
        \left\| e^{\frac{1}{2}s} A_{P,s,a} - A_{P,\infty,a}^{\mathrm{av}}\right\|_{\mathcal{C}^{k+1,\beta+1/2}(\Sigma_a)} &\le C \:
        e^{- \gamma \:s} \left\| P \right\|_{\mathcal{D}^{1,\beta}(\operatorname{Supp}(\d \chi_\infty))}, \\
        \left\| e^{\frac{3}{2}s} B_{P,s,a} - B_{P,\infty,a}^{\mathrm{av}}\right\|_{\mathcal{C}^{k,\beta+1/2}(\Sigma_a)} &\le C \:
        e^{- \gamma \:s} \left\| P \right\|_{\mathcal{D}^{1,\beta}(\operatorname{Supp}(\d \chi_\infty))}.
        \end{align*}
\end{proposition}

\section{Metric perturbations and the averaged coefficient map}
\label{sec:limit-coefficients-metric}
The goal of this section is to show the following proposition:
\begin{proposition}
    \label{prop:phase-family-of-zonal-polynomials}
    Let $V\subset\mathcal{H}_{\mathrm{poly}}(\RR^n)$ be a subspace of dimension $4p$. 
    For a generic choice of $g_s$ on the transition region and for sufficiently large $s$, there is a smooth family of polynomials
    $P_{s,\theta}\in V$, parameterized by $s \gg 0$ and $\theta\in S^1 \subset \CC$, such that for each
    $a\in\{1,\ldots,p\}$
    \[
    A_{P_{s, \theta},s,a}^{\mathrm{av}} = 0, 
    \quad B_{P_{s, \theta},s,a}^{\mathrm{av}} = \theta,
    \quad\text{and}\quad \|P_{s,\theta}\|_{\mathcal{D}^{1, \beta}(\operatorname{Supp}(\d \chi_\infty))} \le Ce^{\frac{3}{2} s} 
    \]
    for some uniform constant $C > 0$.
    Moreover, \(P_{s,-\theta}=-P_{s,\theta}\).
\end{proposition}
The properties of $P_{s, \theta}$ will be important later. Namely, in the next section we will show that for $s \gg 0$ and $\theta \in S^1$, there exists a perturbation $\tilde \Sigma$ of $\Sigma$ and a $\ZT$-harmonic section on $X_s$ that approximates $P_{s, \theta}$ at infinity and has $\tilde \Sigma$ as its branching set. The phase $\theta$ will be chosen after the Nash--Moser correction to make
the Euclidean flux $\lim_{R\to \infty}\int_{S_R^{n-1}}\partial_r u\,dA$ vanish.

The argument will be very similar to the one presented in \cite{salm2024construction}. 
The construction in \cite{salm2024construction} uses a sufficiently large space of linearly independent $L^2$-bounded $\ZZ/2$-harmonic 1-forms to construct a 
specific $L^2$-bounded $\ZZ/2$-harmonic 1-form with properties analogous to those of Proposition \ref{prop:phase-family-of-zonal-polynomials}.
In this section we show that harmonic polynomials can be used in our setting instead.

An important tool in \cite{salm2024construction} is the existence of `Poisson sections'. When applied to our setting, they are as follows:
\begin{lemma}[Coordinate Poisson sections]
\label{lem:coordinate-poisson-sections}
For each component \(\Sigma_a\) there exist \(g_\infty\)-dependent
\(\mathcal I\otimes \CC\)-valued harmonic sections \(G_{A,a}\) and
\(G_{B,a}\), such
that for every \(w\in E^{k+2,\beta}_{\delta,\gamma}(X_\infty)\) whose
Laplacian is compactly supported on $X_{\infty}$, we have
\begin{align*}
A^{\mathrm{av}}_{\infty,a}(w)
=
\int_{X_\infty}\langle G_{A,a},\Delta_\infty w\rangle\,\d\operatorname{vol}_{g_\infty},\quad
B^{\mathrm{av}}_{\infty,a}(w)
=
\int_{X_\infty}\langle G_{B,a},\Delta_\infty w\rangle\,\d\operatorname{vol}_{g_\infty}.
\end{align*}
 
Moreover, on the \(a\)-th cylindrical end,
\[G_{A,a}=c_{A,a}e^{-t/2}e^{-i\phi/2}+O(e^{\gamma t}), \quad
G_{B,a}=c_{B,a}e^{-3t/2}e^{-3i\phi/2}+O(e^{\gamma t})\] as \(t\to-\infty\),
with \(c_{A,a},c_{B,a}\ne0\).  On every
other cylindrical end, these two sections decay exponentially, while on the
Euclidean end $G_{A,a},G_{B,a}=O(r_\infty^{2-n})$.
\end{lemma}

\begin{proof}
We use the cutoff-and-correction construction from
\cite[Section~4]{salm2024construction}.  On the \(a\)-th cylindrical end,
let \(Q\) denote either of the growing model solutions displayed above and
choose a cutoff \(\eta\) which equals \(1\) for \(t\ll0\) and vanishes near
the transition region.  Then \(\Delta_\infty(\eta Q)\) is compactly
supported.  Proposition~\ref{prop:uniform-invertibility} gives a correction
\(H\) satisfying $\Delta_\infty H=-\Delta_\infty(\eta Q)$,
with exponential decay on every cylindrical end and
\(H=O(r_\infty^{2-n})\) on the Euclidean end.  

Thus \(G:=\eta Q+H\) is
harmonic and has the asserted asymptotics.  After normalizing the leading
constant in \(Q\), Green's identity gives the two coefficient formulas.
The boundary integrals vanish on the other cylindrical ends by exponential
decay and at Euclidean infinity by the \(O(r_\infty^{2-n})\) estimate.
\end{proof}

These Poisson sections are useful in the study of the variation of the $A^{\mathrm{av}}_{\infty, a}$ and $B^{\mathrm{av}}_{\infty, a}$ terms. To show this, we need to revisit the work of \cite{salm2024construction} and \cite{He2025}:
Define the perturbation $g_{\infty, \tau} = g_\infty + \tau \: T$, where $|\tau| \ll 1$ and $T$ is a smooth 2-tensor, supported on the transition region of $X_\infty$. Write $u_{P,\infty,\tau} = v_{P,\infty,\tau} + \chi_{\infty} P$ for the solution of the equation $\Delta_{g_{\infty, \tau}} u_{P,\infty,\tau} =0$, so that $u_{P,\infty,0}=u_{P,\infty}$. Denote the $\tau$-derivative of $u_{P,\infty,\tau}$ at $\tau=0$ by $\dot{u}_{P,\infty}$. According to \cite[Proposition 3.3]{HeSalm2026MetricPerturbations},
\begin{equation}
    \label{eq:generic-perturbation-metric}
    \d^{*_{g_{\infty}}} \left(\d \dot{u}_{P,\infty} -
T(\nabla_{g_\infty}u_{P,\infty}, \cdot) + \frac{1}{2} \operatorname{Tr}_{g_{\infty}}(T) \d u_{P,\infty}
\right) = 0.
\end{equation}
This proposition is still valid in our case, as it only depends on the unique solvability of $\Delta_{g_{\infty, \tau}}$ and the smooth dependence of the solution on the metric parameter. Hence, to calculate $\left. \frac{\partial}{\partial \tau} \right|_{\tau=0} A^{\mathrm{av}}_{P,\infty, a}$ we just need to calculate
\begin{align}
        \label{eq:variation-Aav-as-L2-inner-product}
        \left. \frac{\partial}{\partial \tau} \right|_{\tau=0} A^{\mathrm{av}}_{P, \infty, a}
        = \langle G_{A, a}, \d^* (
        T(\nabla_{g_\infty}u_{P,\infty}, \cdot) - \frac{1}{2} \operatorname{Tr}_{g_{\infty}}(T) \d u_{P,\infty}
        ) \rangle_{L^2(X_\infty)}.
\end{align}
{

When $T$ is compactly supported on $M \setminus \Sigma$, He \cite{He2025} showed that
$
\left. \frac{\partial}{\partial \tau} \right|_{\tau=0} A^{\mathrm{av}}_{P, \infty, a}
        = \langle T,  \hat{S}_{G_{A, a}}\rangle_{L^2(X_\infty)},
$
where $\hat{S}_{G} := \frac{1}{2} \left(\d G \otimes \d u_{P, \infty} + \d u_{P, \infty} \otimes \d G -  \langle\d G, \d u_{P, \infty}\rangle_{g_{\infty}} g_\infty \right)$.
Similarly, $\left. \frac{\partial}{\partial \tau} \right|_{\tau=0} B^{\mathrm{av}}_{P,\infty, a}
        = \langle T,  \hat{S}_{G_{B, a}}\rangle_{L^2(X_\infty)}$.

\begin{remark}
\label{rem:real-coordinate-variation-formulas}
The preceding identities are written after complexification.  More
precisely, for $F\in\{\operatorname{Re}A_a,\operatorname{Im}A_a,
\operatorname{Re}B_a,\operatorname{Im}B_a\}$,
let \(G_F\) denote the corresponding real Poisson section
\(\operatorname{Re}G_{A,a}\), \(\operatorname{Im}G_{A,a}\),
\(\operatorname{Re}G_{B,a}\), or \(\operatorname{Im}G_{B,a}\).  The
variation formula is then
\[
\left.\frac{\partial}{\partial\tau}\right|_{\tau=0}F
= \langle G_{F}, \d^* (
        T(\nabla_{g_\infty}u_{P,\infty}, \cdot) - \frac{1}{2} \operatorname{Tr}_{g_{\infty}}(T) \d u_{P,\infty}
        ) \rangle_{L^2(X_\infty)}.
\]
All metric perturbations below are understood in this real sense.  In
particular, $\hat{S}_{G_F}$ is a real symmetric \(2\)-tensor.
\end{remark}

Notice that for a fixed choice of $P$ the variation formula in \eqref{eq:variation-Aav-as-L2-inner-product} is identical to the variational formula used in \cite[Section 4]{salm2024construction}. Hence, repeating the proof of \cite[Proposition 2.3]{He2025} or \cite[Proposition 4.9]{salm2024construction} one concludes:
\begin{lemma}
\label{lem:generic-metric-Aav-and-Bav-nonzero}
For a generic metric on the transition region, 
$A^{\mathrm{av}}_{P,\infty, a} \not = 0$ and $B^{\mathrm{av}}_{P,\infty, a} \not = 0$.
\end{lemma}
%
%
        
Similarly, repeating the proofs of \cite[Lemmas~4.12--4.14]{salm2024construction} one concludes:
\begin{lemma}
\label{lem:generic-surjectivity-coefficient-map}
        Let $p$ be the number of connected components of $\Sigma$ and let $V \subset \mathcal{H}_{\mathrm{poly}}(\RR^n)$ be a subspace of dimension $4p$.
For a generic metric \(g\) sufficiently close to \(g_\infty\) and
agreeing with \(g_\infty\) outside \(\mathcal C_\Sigma\), the map
\begin{align*}
        \Gamma_g \colon V \to \RR^{4p}: P \mapsto& \big(
                \operatorname{Re}(A^{\mathrm{av}}_{P, \infty, a}), 
                \operatorname{Im}(A^{\mathrm{av}}_{P, \infty, a}), 
                \operatorname{Re}(B^{\mathrm{av}}_{P, \infty, a}),
                \operatorname{Im}(B^{\mathrm{av}}_{P, \infty, a})\big)_{1\le a\le p}
\end{align*}
is surjective and hence an isomorphism.
\end{lemma}

By Lemma~\ref{lem:generic-surjectivity-coefficient-map}, we can choose a
sufficiently small smooth symmetric \(2\)-tensor \(h\) on \(X_\infty\),
supported in the interior of \(\mathcal C_\Sigma\), such that
\(\Gamma_{g_\infty+h}\) is an isomorphism. For each \(s\), consider the collar preserving diffeomorphism $\Phi_s$ from section \ref{subsec:model-spaces-metric} and extend
\((\Phi_s^{-1})^*h\) by zero to \(X_s\). Set
\begin{equation}
	\label{eq_perturbed_metric}
	g_\infty':=g_\infty+h,\quad \text{and} \quad g_s':=g_s+(\Phi_s^{-1})^*h.
\end{equation}
Thus \(g_\infty'\) and \(g_s'\) carry the same perturbation under the
identification of the transition regions and agree with the original metrics
on all ends. We henceforth replace \(g_\infty\) and \(g_s\) by
\(g_\infty'\) and \(g_s'\), retaining the notation \(g_\infty\) and \(g_s\),
and use the same notation for the associated coefficients. Write
\(\Gamma:=\Gamma_{g_\infty}\) for the resulting limiting coefficient map.
Since \(h\) is fixed, compactly supported in \(\mathcal C_\Sigma\), and
transported to \(\mathcal C_{\Sigma,s}\) by \(\Phi_s\), the estimates proved in
Sections~\ref{sec:long-neck-elliptic} and~\ref{sec:asymptotic-expansions}
remain uniform in \(s\).

\begin{proof}[Proof of Proposition \ref{prop:phase-family-of-zonal-polynomials}]
For sufficiently large \(s\), define
\[\Gamma_s\colon V\longrightarrow\CC^{2p}:
P\mapsto
\bigl(e^{s/2}A^{\mathrm{av}}_{P,s,a},
e^{3s/2}B^{\mathrm{av}}_{P,s,a}\bigr)_{1\le a\le p}.\]
Equip \(V\) with the fixed norm $ \|P\|_V:=
\|P\|_{\mathcal D^{1,\beta}
(\operatorname{Supp}(\d\chi_\infty))}.$ Equip
\(\CC^{2p}\cong\RR^{4p}\) with its Euclidean norm, and use the induced
operator norms $\|\cdot\|_{\mathrm{op}}$.
Proposition~\ref{prop:decay-rate_A_and_B-in-limit} gives $\|\Gamma_s-\Gamma\|_{\mathrm{op}}
\le Ce^{-\gamma s}.$
Since \(\Gamma\) is an isomorphism,
$\|\Gamma^{-1}(\Gamma_s-\Gamma)\|_{\mathrm{op}}
\le 1/2$ for \(s \gg 0\).

The factorization
$\Gamma_s
=\Gamma\bigl(I+\Gamma^{-1}(\Gamma_s-\Gamma)\bigr)$
and the Neumann-series estimate therefore imply that \(\Gamma_s\) is
invertible and $\|\Gamma_s^{-1}\|_{\mathrm{op}}
\le2\|\Gamma^{-1}\|_{\mathrm{op}}.$

For \(\theta\in S^1\), set
$P_{s,\theta}:=
e^{3s/2}\Gamma_s^{-1}(0,\theta,\ldots,0,\theta)$.
Then
\[
\|P_{s,\theta}\|_V
\le2e^{3s/2}\|\Gamma^{-1}\|_{\mathrm{op}}
\|(0,\theta,\ldots,0,\theta)\|
\le Ce^{3s/2},
\]
uniformly in \(s\) and \(\theta\).  By the definition of \(\Gamma_s\),
for every \(a\in\{1,\ldots,p\}\), $A^{\mathrm{av}}_{P_{s,\theta},s,a}=0$ and
$B^{\mathrm{av}}_{P_{s,\theta},s,a}=\theta$.
For fixed $s$, the map
$\theta\mapsto e^{3s/2}\Gamma_s^{-1}(0,\theta,\ldots,0,\theta)$ is the
restriction to $S^1\subset\CC$ of a real-linear map $\CC\to V$.  Hence
$\theta\mapsto P_{s,\theta}$ is smooth and
$P_{s,-\theta}=-P_{s,\theta}$.
\end{proof}

\section{Uniform Nash--Moser deformation}
\label{sec:nash-moser}

This section develops the Nash--Moser theory for harmonic sections
\(u_{P,s}\) with vanishing averaged $A$-coefficients and nowhere-vanishing
averaged $B$-coefficients, and constructs nondegenerate \(\ZT\)-harmonic
sections on $\RR^n$, using the deformation theory of multivalued harmonic
functions \cite{Donaldson2021,salm2024construction}, together with 
\cite{DonaldsonFabian25Calabi,Zehnder17generalized,HeSalm2026MetricPerturbations}.

\subsection{Moved singular sets and uniform analytic control}
From now on, fix $s \gg 0$ and $\theta \in S^1$. From Proposition \ref{prop:phase-family-of-zonal-polynomials} we get a harmonic section $u_{P_{s,\theta}, s}$ which is almost a $\ZZ/2$-harmonic 1-form and approximates $P_{s,\theta}$ near infinity. In this section we will perturb the branching set $\Sigma$, turning this almost $\ZZ/2$-harmonic 1-form into a $\ZZ/2$-harmonic 1-form.

Explicitly, let $\Sigma_0:=\Sigma$ and for a sufficiently small $\nu \in \Gamma(N)$ let $\Sigma_\nu:=\operatorname{graph}_{\Sigma_0}(\nu)$. Choose a smooth tame family of diffeomorphisms $\Psi_\nu:\RR^n\to\RR^n$ from \cite[Proposition 4.2]{Donaldson2021} for which $\Psi_\nu(\Sigma_0) =\Sigma_\nu$, $\Psi_0=\operatorname{Id}$, and $\Psi_\nu=\operatorname{Id}$ outside the boundary region. Let $\mathcal I_0=\mathcal I\to\RR^n\setminus\Sigma_0$ be the reference flat line bundle and set $\mathcal I_\nu:=(\Psi_\nu^{-1})^*\mathcal I_0$.

By Proposition~\ref{prop:existence-harmonic-sections-with-polynomial-growth}, for every $P\in\mathcal H_{\mathrm{poly}}(\RR^n)$ and sufficiently small $\nu$, there exists a unique
$\mathcal I_\nu$-valued harmonic section $u_{P,s,\nu}$ such that $u_{P,s,\nu}=P+O(r_\infty^{\delta_*})$ on $\mathcal{A}_\infty$.

Notice that $\Psi_\nu^*u_{P,s,\nu}\in \Gamma(\mathcal{I}_0)$. According to \cite{Donaldson2021},  there is a local expansion of $\Psi_\nu^*u_{P,s,\nu}$ in the fixed normal coordinate $z=re^{i\phi}$ of $\Sigma_0$ of the form
\begin{equation}
	\label{eq:moved-local-expansion}
	\begin{split}
		\Psi_\nu^*u_{P,s,\nu}
		={}&\operatorname{Re}\left(A_s(\nu,P)z^{1/2}
		+B_s(\nu,P)z^{3/2}\right) \\
		&-\frac12\operatorname{Re}\left(A_s(\nu,P)z^{1/2}\right)
		\operatorname{Re}(\bar{\mu}_\nu z)+O(r^{5/2}),
	\end{split}
\end{equation}
where $\mu_\nu$ is the mean-curvature vector of
$\Sigma_\nu$. We use \eqref{eq:moved-local-expansion} to define $A_s(\nu,P)$ and $B_s(\nu,P)$. Since $\Sigma_0$ has trivial normal bundle, the coefficients $A_s(\nu,P)$ and $B_s(\nu,P)$ may be viewed as complex-valued functions on $\Sigma_0$.

Recall that, for fixed $s$, the map
$\nu\mapsto A_s(\nu,P_s)$ is called \emph{tame} if there is an integer
$r\ge0$ such that, for every $m\ge0$,
\[
\|A_s(\nu,P_s)\|_{\mathcal C^m(\Sigma_0)}
\le
C_{m,s}\bigl(1+\|\nu\|_{\mathcal C^{m+r}(\Sigma_0)}\bigr)
\]
for $\nu$ in a fixed neighbourhood of the zero section. The constants
$C_{m,s}$ are called the tame constants. The map is
smooth tame if the corresponding estimates hold for all its
derivatives. The same terminology applies to the map
$\nu\mapsto B_s(\nu,P_s)$ and the linear map
$\sigma\mapsto\mathscr P_s\sigma$ defined in \eqref{def:D-N-operator} below.

The next lemma shows that the derivative loss and the tame
constants can be chosen independently of $s$.
\begin{lemma}
	\label{lem:uniform-moved-coefficients}
	Let $s > 0$, $\theta \in S^1$ and let $P_{s,\theta} \in \mathcal H_{\mathrm{poly}}(\RR^n)$ be as in Proposition \ref{prop:phase-family-of-zonal-polynomials}.
	There is a tame neighbourhood $\mathcal U\subset\Gamma(N)$ of the zero section,
	independent of $s$ and $\theta$, such that $u_{P_{s, \theta},s,\nu}$ is uniquely defined
	for every $s \gg 0$, $\theta \in S^1$ and $\nu\in\mathcal U$. Moreover, the maps
	\[
	\nu\longmapsto A_s(\nu,P_{s, \theta}),
	\qquad
	\nu\longmapsto B_s(\nu,P_{s, \theta}),\qquad \sigma\mapsto\mathscr P_{s}\sigma
	\]
are smooth tame, with tame
constants independent of $s$ and $\theta$.
\end{lemma}

\begin{proof}
	The argument of \cite[Lemma~3.3]{salm2024construction} applies directly.
	The uniformity in $s$ follows from Proposition
	\ref{prop:uniform-invertibility}, the fact that $u_{P_{s, \theta},s}$ is uniformly bounded in $\mathcal{B}_{\Sigma}$ and the fact that the deformation is supported in $\mathcal{B}_{\Sigma}$. See also
	\cite[Remark~3.4 and Lemma~3.6]{salm2024construction}.
\end{proof}

\subsection{Donaldson's linearization and the correction theorem}
Next we revisit the variation of $A_s$ with respect to the variation of the branching set.
For this, we need to define Donaldson's Dirichlet-to-Neumann type operator. 
\begin{proposition}[{\protect\cite[Theorem~4]{Donaldson2021}}]
	For every $s\ge1$ and $\sigma\in\Gamma(N^{1/2})$, there is a unique
	$\mathcal I_0\otimes\CC$-valued harmonic section $Q_s(\sigma)$ with zero asymptotic polynomial such
	that
	\[
	Q_s(\sigma)
	=
	\sigma z^{-1/2}
	+
	(\mathscr P_s\sigma)z^{1/2}
	+
	O(r^{3/2})
	\]
	near $\Sigma_0$.
	The coefficient of $z^{1/2}$ defines a linear map
	\begin{align}\label{def:D-N-operator}
	\mathscr P_s:\Gamma(N^{1/2})\longrightarrow\Gamma(N^{-1/2}),
	\end{align} 
	which is called the Dirichlet-to-Neumann operator.
\end{proposition}

Donaldson's variation formula \cite[Section~5.1]{Donaldson2021} gives
\begin{equation}
	\label{eq:donaldson-derivative-formula}
	\begin{aligned}
		D_\nu A_s(\nu,P)[\dot\nu]
		={}&
		\frac32B_s(\nu,P)\dot\nu+
		\underbrace{\left(
			-\frac12\mathscr P_s
			\bigl(A_s(\nu,P)\dot\nu\bigr)
			-\frac12\langle\mu_\nu,\dot\nu\rangle A_s(\nu,P)
			\right)}_{
			\mathcal Q_s\left(\nu,A_s(\nu,P),\dot\nu\right)
		}.
	\end{aligned}
\end{equation}
Here $\mathcal Q_s$ is smooth tame, bilinear in its last two variables. Donaldson's derivation of this equation can be applied directly, so \eqref{eq:donaldson-derivative-formula} is also true in our situation. Even more, $\mathcal Q_s$ has tame constants independent of $s$ and $\theta$. Indeed, 
to construct $\mathscr P_s$ Donaldson found a section $\tau_\sigma \in \Gamma(N^{\frac12})$ such that $\Delta_s (\sigma z^{-1/2} + \tau_\sigma \bar{z} z^{1/2}) = \mathcal{O}(|z|^{3/2})$. For this, one only needs to know the metric near the branching set. Next, to find $Q_s$, Donaldson needed to solve 
$
\Delta_s u = - \Delta_s (\chi\cdot(\sigma z^{-1/2} + \tau_\sigma \bar{z} z^{1/2})),
$
for some cut-off function $\chi$ that is supported in a neighbourhood of $\Sigma_0$. As $\Delta^{-1}_s$ is a tame map with uniform tame estimates \cite[Remark 3.4]{salm2024construction}, we get uniform tame estimates on $u$. $\mathscr P_s(\sigma)$ is just the $A$-term of $u$, which must be also uniformly tame.

If $B_s(\nu,P)$ is nowhere zero, define $R_{s,P}(\nu)\xi
:=
\frac23B_s(\nu,P)^{-1}\xi$.
Then \eqref{eq:donaldson-derivative-formula} gives
\[
\begin{aligned}
	R_{s,P}(\nu)D_\nu A_s(\nu,P)[\dot\nu]
	&=
	\dot\nu+
	R_{s,P}(\nu)
	\mathcal Q_s\left(\nu,A_s(\nu,P),\dot\nu\right),\\
	D_\nu A_s(\nu,P)[R_{s,P}(\nu)\xi]
	&=
	\xi+
	\mathcal Q_s\left(
	\nu,A_s(\nu,P),R_{s,P}(\nu)\xi
	\right).
\end{aligned}
\]
Thus $R_{s,P}(\nu)$ is a two-sided inverse up to quadratic errors bilinear in $A_s(\nu,P)$. The following consequence of
\cite[Proposition~A.2]{HeSalm2026MetricPerturbations} and \cite[Theorem~A.11]{DonaldsonFabian25Calabi} records the
uniform existence estimate used below.

\begin{theorem}
	\label{thm:nash-moser-correction}
	Let $s > 0$, $\theta \in S^1$ and let $P_{s,\theta} \in \mathcal H_{\mathrm{poly}}(\RR^n)$ be as in Proposition \ref{prop:phase-family-of-zonal-polynomials}.
	Then there is  a fixed neighbourhood
	$\mathcal U_0\subset \mathcal{C}^\infty(\Sigma_0;\mathbb C)$ around the zero section, such that for all
	$s \gg 0$ and $\theta \in S^1$ there is a unique $\nu_{s, \theta}\in\mathcal U_0$ 
	such that
	$A_s(\nu_{s,\theta},P_{s, \theta})=0$.
	Moreover, $u_{P_{s, \theta},s,\nu_{s,\theta}}$ is a nondegenerate $\ZT$-harmonic section.
\end{theorem}

\begin{proof}
	By Proposition \ref{prop:decay-rate_A_and_B},
	\begin{align*}
		\| A_s(0, P_{s, \theta})\|_{\mathcal{C}^{k + 1, \beta + \frac{1}{2}}(\Sigma_0)} &= \mathcal{O}(e^{- \kappa\: s}), \\
		\| B_s(0, P_{s, \theta}) - \theta \|_{\mathcal{C}^{k, \beta + \frac{1}{2}}(\Sigma_0)} &= \mathcal{O}(e^{- \kappa\: s}),
	\end{align*}
	where $\kappa = \sqrt{\frac{1}{4} + \frac{\mu_1}{\varepsilon^2}} - \frac{3}{2} > 0$.
	Moreover, according to Lemma \ref{lem:uniform-moved-coefficients} there is some
	fixed derivative loss $d$, and a family of constants $C_k$, independent of $s$, such that for all $\nu \in \mathcal{U}$
	\[
	\begin{aligned}
		\|B_s(\nu,P_{s, \theta})-\theta\|_{\mathcal{C}^{k, \beta + \frac{1}{2}}(\Sigma_0)}
		&\le
		\|B_s(\nu,P_{s, \theta})-B_s(0,P_{s, \theta})\|_{\mathcal{C}^{k, \beta + \frac{1}{2}}(\Sigma_0)}
		+\|B_s(0,P_{s,\theta})-\theta\|_{\mathcal{C}^{k, \beta + \frac{1}{2}}(\Sigma_0)} \\
		&\le
		C_k\|\nu\|_{\mathcal{C}^{k+d, \beta + \frac{1}{2}}(\Sigma_0)}
		+C_k e^{-\kappa s}.
	\end{aligned}
	\]
	As $|\theta| = 1$, we can find a neighbourhood $\mathcal{U}_0 \subset \mathcal C^\infty(\Sigma_0, \CC)$, such that $R_{s, P_{s, \theta}}$ is a uniformly smooth tame map on $\mathcal{U}_0$.
	Using $\| A_s(0, P_{s, \theta})\|_{\mathcal{C}^{k + 1, \beta + \frac{1}{2}}(\Sigma_0)} = \mathcal{O}(e^{- \kappa\: s})$, the result from \cite[Theorem~A.11]{DonaldsonFabian25Calabi} implies there is a $\nu_{s, \theta} \in \mathcal{U}_0$ such that 
	$$
	A_s(\nu_{s, \theta}, P_{s, \theta}) = 0.
	$$

	By \cite[Theorem~III.3.3.3]{Hamilton1982} we can shrink $\mathcal{U}_0$ to make $\nu_{s, \theta}$ unique.
	Following the details of this proof, the shrinking can be done uniformly in $s$ and $\theta$.
\end{proof}

\begin{proof}[Proof of Theorem \ref{theorem:euclidean-realization}]
The construction in Theorem \ref{thm:nash-moser-correction} produces a small normal deformation $\Sigma'$ of $\Sigma$. Recall, there is a compactly supported diffeomorphism $\Phi$ of $\mathbb R^n$ such that $\Phi(\Sigma)=\Sigma'$. Pulling back the metric, line bundle, and harmonic section by $\Phi$ gives the
formulation as in the theorem with singular set exactly $\Sigma$. Since $\Phi$ is the
identity near infinity, this pullback preserves the Euclidean end and the
asymptotic data of the model.
\end{proof}

For every choice of $s \gg 0$ and $\theta$, the Nash--Moser theorem in \cite{DonaldsonFabian25Calabi} yields a choice of $\nu_{s, \theta}$. Viewing $\theta$ as a tame parameter, we can use \cite[Theorem III.3.3.1]{Hamilton1982} to conclude
\begin{corollary}
	\label{cor:parametric-nash-moser-correction}
	For every $s \gg 1$, the map $
	\theta \mapsto \nu_{s, \theta}
	$
	is a smooth tame map.
\end{corollary}

\begin{proof}
For each $\theta\in S^1$, let $\nu_{s, \theta}\in\mathcal U_0$ be the unique
solution given by Theorem~\ref{thm:nash-moser-correction}. Fix
$\theta_0\in S^1$. Applying Hamilton's implicit-function theorem with
quadratic error
\cite[Theorem~III.3.3.1]{Hamilton1982} at the exact solution
$(\nu_{s, \theta_0},\theta_0)$ gives a neighbourhood $W \subset S^1$ of $\theta_0$ and a smooth tame map $\widetilde\nu_s:W\longrightarrow
\mathcal C^\infty(\Sigma_0;\mathbb C)$
such that
\[
\widetilde\nu_s(\theta_0)=\nu_{s, \theta_0},
\qquad
A_s(\widetilde\nu_s(\theta),P_{s, \theta})=0.
\]
After shrinking $W$, we may assume that
$\widetilde\nu_s(\theta)\in\mathcal U_0$. The uniqueness in
Theorem~\ref{thm:nash-moser-correction} then gives
$\widetilde\nu_s(\theta)=\nu_{s,\theta}$ for $\theta\in W$. Thus,
$\theta\mapsto\nu_{s, \theta}$ is smooth.
\end{proof}

\subsection{Vanishing of the Euclidean flux}

Fix a parallel unit trivialization of $\I$ on $\mathcal A_\infty$.  For a
harmonic section $u$ with polynomial asymptotic term $P$, define its
Euclidean flux by
\[
	\Flux(u):=\int_{S_R^{n-1}}\star_{g_{\mathrm{Eucl}}}\d u
	=\int_{S_R^{n-1}}\partial_r u\,\mathrm dA,
\]
where $S_R^{n-1}$ lies in the Euclidean end.  The integral is independent of
$R$ when $R$ is sufficiently large.  The exterior spherical-harmonic expansion gives
\[
	u=P+\frac{\Flux(u)}{(2-n)\lvert S^{n-1}\rvert}r^{2-n}
	+O(r^{1-n}).
\]
Thus $\Flux(u)=0$ precisely when $\star_{g_{\mathrm{Eucl}}}\d u$ is exact
near infinity.  Reversing the trivialization only changes the sign of both $u$
and $\Flux(u)$, so the condition $\Flux(u)=0$ is intrinsic.

We aim to find a $\ZT$-harmonic function with zero Euclidean flux. However, the Nash--Moser correction does not fix the flux. For this reason, we considered a 2-parameter family $P_{s, \theta}$ in 
Section~\ref{sec:limit-coefficients-metric}. We use the fact that this 2-parameter family satisfies $P_{s,\theta}=-P_{s,-\theta}$ to show the following:

\begin{theorem}
	\label{thm:flux-normalization}
	For every $s\gg 1$, there are $\theta_s\in S^1$ and a deformation $\nu_s$ of $\Sigma_0$ such that, with
	$P_s:=P_{s,\theta_s}\in V$, the function $u_{P_s,s,\nu_s}$ is nondegenerate,
	$\ZT$-harmonic, and has zero Euclidean flux.
\end{theorem}

\begin{proof}
	Corollary~\ref{cor:parametric-nash-moser-correction} gives a
	smooth family of unique small normal deformations $\nu_{s,\theta}$ such that
	$A_s(\nu_{s,\theta},P_{s,\theta})=0$ with a nowhere-vanishing $B$-coefficient. Set $u_{s,\theta}:=u_{P_{s,\theta},s,\nu_{s,\theta}}$.
Since $P_{s,-\theta}=-P_{s,\theta}$, uniqueness gives
\[
\nu_{s,-\theta}=\nu_{s,\theta},
\qquad
u_{s,-\theta}=-u_{s,\theta}.
\]

The harmonic sections depend smoothly on $\theta$, so the map
	$\theta\mapsto\Flux(u_{s,\theta})$ is continuous and odd, and hence
	vanishes at some $\theta_s\in S^1$. Setting
	$P_s:=P_{s,\theta_s}$ and $\nu_s:=\nu_{s,\theta_s}$ proves the theorem.
\end{proof}

}


\section{Gluing Euclidean models into closed manifolds}
\label{sec:compact-zero-flux-gluing}
We apply the Calabi surgery construction developed in \cite{HeChenYan} to the zero-flux Euclidean models obtained above. 
The gluing argument has two main stages, matching and insertion. In the first subsection, 
we identify the first nonzero homogeneous term in the Taylor expansion of a local potential on the closed manifold and formulate 
the corresponding matching condition for the asymptotic polynomial of the Euclidean model. 
In the second subsection, we rescale a matched model and insert it into the closed manifold. 
The matching condition ensures that the interpolated $1$-form remains nonvanishing on the transition annulus, 
while the zero-flux condition removes the cohomological obstruction to interpolating the corresponding Hodge duals by a closed $(n-1)$-form. 
In the final subsection, we apply the resulting insertion theorem to linear and higher-degree models.

\subsection{Polynomial matching}
\label{subsec:compact-matching-data}
Let $(M^n,g)$ be a closed oriented Riemannian manifold, where
$n\geq 3$.  Let $(v,\Sig_M,\I_M)$ be a nonzero $\ZT$-harmonic $1$-form on $M$, in the
sense of \cite{HeChenYan}.  Fix a
point $q\in M\setminus\Sig_M$ and a geodesic ball about $q$ that is disjoint
from $\Sig_M$.  The flat line bundle $\I_M$ is trivial on this ball.  After
choosing a trivialization of $\I_M$ on this ball, write $v=\d f_q$, where $f_q$ is a single-valued harmonic function and
$f_q(q)=0$.

Unique continuation shows that $f_q$ has finite vanishing order at $q$.  Let
$P_q$ be its first nonzero homogeneous Taylor term, and set
$m=\deg P_q\geq 1$.  In oriented orthonormal geodesic coordinates $x$
centered at $q$, the polynomial $P_q$ is nonzero and harmonic with respect to
the Euclidean metric determined by $g(q)$.  Moreover,
$\nabla^j(f_q-P_q)=\mathcal{O}(|x|^{m+1-j})$ for $0\leq j\leq 2$.
Changing the local trivialization of $\I_M$ replaces $P_q$ by $-P_q$.  Changing the
oriented orthonormal coordinates acts on $P_q$ by an element of $SO(n)$.
Thus the polynomial data are naturally defined only up to sign and an
oriented orthogonal change of coordinates.

\begin{definition}
\label{def:admissible-matching-point}
The point $q$ is a \textbf{matching point} if $\d P_q$ has no zero in
$\RR^n\setminus\{0\}$.  We call $P_q$ the matching polynomial.
\end{definition}

Homogeneity gives $|\d P_q(x)|\geq c|x|^{m-1}$ for some $c>0$.  If
$m\geq2$, the matching condition and the Taylor estimate show that $q$ is an
isolated zero of $v$.  Every regular point of $v$ is a matching point with
$m=1$ and a linear matching polynomial.

We next formulate the corresponding condition on the Euclidean side.  Only
the highest-degree homogeneous term of the asymptotic polynomial must agree
with $P_q$.  The lower-degree terms become lower-order errors after
rescaling.

\begin{definition}
\label{def:euclidean-model-matching}
Let $g_e$ denote a smooth Riemannian metric on
$\RR^n$ that is Euclidean outside a compact set.  Let
$(u,\Sig_{\RR^n},\I_{\RR^n})$ be a nondegenerate $\ZT$-harmonic
function on $(\RR^n,g_e)$ with zero flux, and let $P$ be its asymptotic
polynomial.  We call $(u,\Sig_{\RR^n},\I_{\RR^n})$ on $(\RR^n,g_e)$ a
\emph{zero-flux Euclidean model}.  Let $P_q$ be a matching polynomial of
degree $m$.  We say that the model \emph{matches} $P_q$ if $P$ has degree
$m$ and there are
$c\in\RR^*$ and $A\in SO(n)$ such that $P^{(m)}(x)=cP_q(Ax)$, where
$P^{(m)}$ is the highest-degree homogeneous term of $P$.
\end{definition}
In particular, if $P$ has degree $1$, then the model matches the matching
polynomial at every regular point of $v$.

For a matched pair, rotate the entire Euclidean model and multiply $u$ by a
nonzero real constant so that $P^{(m)}=P_q$.  These operations preserve the
zero-flux condition, so we may assume this normalization throughout the
gluing construction.  Write $P=P_q+P_{<m}$, where $P_{<m}$ is harmonic and
has degree at most $m-1$.  Thus
$\nabla^j(u-P_q)=O(r_\infty^{m-1-j})$ for $0\leq j\leq 2$.

\subsection{Zero-flux insertion}
\label{subsec:zero-flux-gluing-theorem}

Given a matching point and a Euclidean model matching its polynomial, we
prove that the model can be inserted into the closed manifold.
Before we prove this, we first recall some elementary properties of differential forms:

\begin{lemma}
\label{lem:calabi-metric-preparations}
Let $M$ be a manifold of dimension $n$ and let $\alpha \in \Omega^1(M)$ and let $\beta \in \Omega^{n-1}(M)$ such that 
$\alpha \wedge \beta$ is nowhere vanishing.
There exists a unique $N \in \Gamma(TM)$ such that $\iota_N(\alpha\wedge \beta) = \beta$.
Moreover, for the canonical projection map $\pi(v) := v - \frac{\alpha(v)}{\alpha(N)}N$ from $\Gamma(TM)$ to the bundle $H := \ker \alpha$, $\beta = \pi^* \mu$ for some volume form $\mu$ on $H$.
\end{lemma}
\begin{proof}
  We first claim that at every $p \in M$, there is a local coframe $\{e^1, \ldots, e^n\}$ such that 
  \begin{equation}
    \label{eq:nice-local-coordinates-between-forms}
    \alpha = e^1 \quad \text{ and } \quad \beta = e^2 \wedge \ldots \wedge e^n.
  \end{equation}
  Indeed, pick a Riemannian metric $g$ and orientation on $M$. Then $\beta = *_g \sigma$ for some $\sigma \in \Omega^1(M)$.
  As $\alpha \wedge \beta$ is nowhere vanishing, both $\alpha$ and $\beta$ are nowhere vanishing and so $\sigma$ is nowhere vanishing.
  Hence, let $\tilde e^1 = \frac{\sigma}{|\sigma|_g}$ and extend it into a local orthonormal coordinate chart $\{\tilde e^1, \ldots, \tilde e^n\}$.
  In this frame, 
  $$
  \beta = *_g \sigma = |\sigma| \cdot \tilde e^2 \wedge \ldots \wedge \tilde e^n.
  $$
  We claim that the forms $\{e^1 := \alpha, e^2 := |\sigma| \: \tilde e^2, e^3 := \tilde{e}^3, \ldots, e^n := \tilde{e}^n \}$ form a local coframe satisfying \eqref{eq:nice-local-coordinates-between-forms}. The only non-obvious thing is to show that $\{e^1, \ldots e^n \}$ are linearly independent. If they weren't, then $e^1 \wedge \ldots \wedge e^n$ would vanish somewhere, but
  $$
  e^1 \wedge \ldots \wedge e^n = \alpha \wedge \beta \not = 0.
  $$

  Using this claim, let ${e_1, \ldots, e_n}$ be the dual frame of $\{e^1, \ldots, e^n\}$. Notice that $N:= e_1$ is the unique solution to $\iota_N (\alpha \wedge \beta) = \beta$. As $\alpha(N) = 1$, $\pi$ is indeed a projection map from $TM$ to $\ker \alpha$. Also notice that $\pi(N) = 0$.

  To show $\beta = \pi^* \mu$ for some volume form $\mu$ on $\ker \alpha$, it is sufficient to show that $\iota_N \beta = \iota_N \pi^* \mu = 0$.
  Because $\iota_N(\alpha \wedge \beta) = \beta$,
  $
  \iota_N \beta = \iota_N \iota_N (\alpha \wedge \beta) = 0.
  $
  Simultaneously, 
  $
  \iota_N \pi^* \mu = \mu(\pi(N), \ldots) = \mu(0, \ldots) = 0,
  $
  and we conclude the proof.
\end{proof}

\begin{proposition}
\label{prop:calabi-metric-realization}
Let $M$ be an oriented manifold of dimension $n$ and let $\alpha \in \Omega^1(M)$ and let $\beta \in \Omega^{n-1}(M)$ such that 
$\alpha \wedge \beta > 0$.
Let $\pi\colon \Gamma(TM) \to \Gamma(\ker \alpha)$ be as defined in Lemma \ref{lem:calabi-metric-preparations}.
Then $g$ is a Riemannian metric on $M$ such that $*_g \alpha = \beta$ if and only if 
there exists a metric $h$ on the bundle $\ker \alpha$ such that $\pi^* \operatorname{vol}_h = f \cdot \beta$ for some nowhere vanishing function $f$ and
\begin{equation}
  \label{eq:canonical-form-metric-g}
  g = f^2 \cdot \alpha^2 + \pi^*h.
\end{equation}
\end{proposition}
\begin{proof}
Let $\{e^1, \ldots, e^n\}$ be a local oriented co-frame such that \eqref{eq:nice-local-coordinates-between-forms} holds.
Let $\{e_1, \ldots, e_n\}$ be its dual frame.
Assume that $g$ is a Riemannian metric such that $* \alpha = \beta$.
Then for each index $i$
\begin{align*}
  e^i \wedge * \alpha =& e^i \wedge * e^1 = g(e^i, e^1) \: \operatorname{vol}_g \\
  =& e^i \wedge \beta = \begin{cases}
    e^1 \wedge \ldots \wedge e^n & \text{if } i = 1 \\
    0 & \text{else.}
  \end{cases}
\end{align*}
Hence, $e^1$ is orthogonal to $\langle e^2, \ldots, e^n \rangle$ and we can write the metric $g$ locally as a block matrix. Dualizing this shows that $e_1$ is orthogonal to $\langle e_2, \ldots, e_n \rangle = \ker \alpha$. By noticing that $\pi(e_1) = 0$ and $\pi(e_i) = e_i$ for $i \not = 1$,
we conclude that there is a metric $h$ on $\ker \alpha$ and a positive smooth function $f$ such that
$$
g = f^2 \cdot \alpha^2 + \pi^* h.
$$
To determine $f$, use that
\begin{equation}
  \label{eq:relationship-staralpha-beta}
  \beta = * \alpha 
  = \frac{1}{f} \pi^* \operatorname{vol}_h.
\end{equation}
This proves one direction of the claim. For the converse, let $h$ be a metric on $\ker \alpha$ and let $g$ be defined in \eqref{eq:canonical-form-metric-g}.
Notice that if 
$$
g(X,X) = f^2 \cdot \alpha(X)^2 + h(\pi(X), \pi(X)) = 0,
$$
for some $X \in \Gamma(TM)$, then $\alpha(X) = 0$ and $\pi(X) = 0$. This implies $X = 0$ and so $g$ is positive definite.
Hence, $g$ is a Riemannian metric on $M$. Equation \eqref{eq:relationship-staralpha-beta} implies that $*_g \alpha = \beta$.
\end{proof}
Proposition \ref{prop:calabi-metric-realization} will be pivotal in the gluing of our Euclidean model into a compact manifold. Namely, instead of solving $* \alpha = \beta$, it is sufficient to smoothly interpolate $\alpha$, $\beta$ and $h$. This can be accomplished using smooth cutoff functions.

\begin{proof}[Proof of Theorem \ref{theorem:calabi-insertion}]
  Consider the geodesic coordinates near $q$ in
which $P_q$ was defined.  Let $g_0$ be the Euclidean metric determined by
$g|_q$.  Use the normalization fixed above and write the asymptotic
polynomial of $u$ as $P=P_q+P_{<m}$.

Choose $R\gg1$ so that $\Sig_{\RR^n}$ and the non-Euclidean part of $g_e$
lie in $B_{R/4}(0)$.  We also require the asymptotic expansion of $u$ and the
exactness of $\star_{g_e}\d u$ to hold outside $B_{R/2}(0)$.  

\medskip
\noindent
\textbf{Rescaled model metric:} Next we will rescale the Euclidean model such that the annulus $\{x \in \RR^n\colon |x|_{g_0} \in (R/2, R) \}$ will be mapped to the annulus $A_\delta = \{x \in \RR^n \colon |x|_{g_\lambda} \in (\delta/2, \delta) \}$.
Hence, set $\lambda=\delta/R$ and consider the new metric $g_\lambda = \lambda^2 g_e$.
The geodesic coordinates $\{x_i\}$ of $g_0$ are not geodesic on $g_\lambda$, hence define $y = \lambda\: x$ and let $u_\lambda = \lambda^m u $.
The homogeneity of $P_q$ gives $\lambda^mP_q(x)=P_q(y)$ and so the leading order behaviour of $u_\lambda$ is still $P_q(y)$.

To estimate the lower order behaviour of $u_\lambda$, notice that $\nabla^j_{g_0} (u - P_q) = \mathcal{O}_{g_0}(R^{m-1-j})$ on $A_\delta$ for all $j \in \NN$. Therefore, $\nabla^j_{g_\lambda} (u_\lambda - P_q) = \mathcal{O}_{g_\lambda}(\lambda^{m-j} R^{m-1-j}) = \mathcal{O}_{g_\lambda}(\delta^{m-j} R^{-1})$ on $A_\delta$.

\medskip
\noindent
\textbf{Local estimates:}
Simultaneously, we have a $\ZZ/2$-harmonic 1-form $v = \d f_q$ on $B_{\delta}(q) \subset (M, g)$ such that $\nabla^j_g(f_q-P_q)=\mathcal{O}_g(|x|^{m+1-j})$. Because $g = g_0 + \mathcal{O}_{g_0}(|x|^2)$, $\nabla^j_{g_0}(f_q-P_q)=\mathcal{O}_{g_0}(\delta^{m+1-j})$ on the annulus $B_{\delta}(q) \setminus \overline{B_{\delta/2}(q)}$.

\medskip
\noindent
\textbf{Connected sum:}
With these estimates we now consider the connected sum of $M$ and $\RR^n$ at $q$. Explicitly, we consider the disjoint union of $M \setminus B_{\delta/2}(q) \subset (M,g)$ and $B_{\delta}(0) \subset (\RR^n,g_\lambda)$, and we identify the annulus $B_{\delta}(q) \setminus \overline{B_{\delta/2}(q)} \subset (M,g)$ with $A_\delta$ we defined before. Now set $\Sig_\delta=\Sig_M\sqcup \Sig_{\RR^n}$. 
As both line bundles are trivial on the respective annuli, we can glue them into a new flat line bundle $\I_\delta$ on $M\setminus\Sig_\delta$.

\medskip
\noindent
\textbf{Global 1-form:}
Now we interpolate the potentials and we define a global multivalued 1-form.  Choose a cutoff $\chi$ on $A_\delta$
that is $1$ near the outer boundary and
$0$ near the inner boundary, with $|\nabla^j\chi|\leq C_j\delta^{-j}$.
Set $$f_\delta=\chi f_q+(1-\chi)u_\lambda \quad \text{and} \quad v_\delta=\d f_\delta \quad \text{on } A_\delta.$$
Extend $v_\delta$ by $v$ outside
$B_\delta(q)$ and by $\d u_\lambda$ inside
$B_{\delta/2}(q)$.  The definitions agree near both boundary components, so $v_\delta$ is closed globally.

Since
$v_\delta-\d P_q=\chi\d(f_q-P_q)+(1-\chi)\d(u_\lambda-P_q)
+(f_q-u_\lambda)\d\chi$ on $A_\delta$, the estimates for $f_q-P_q$,
$u_\lambda-P_q$, and $\d\chi$ give
\[
 \nabla^j_{g_\lambda} (v_\delta-\d P_q) = \mathcal{O}_{g_\lambda} \! \left(
\delta^{m-1-j}(\delta+R^{-1})
 \right)
 \quad \text{on }A_\delta.
\]

\noindent
\textbf{Global (n-1)-form:}
Next we construct a closed $\omega_\delta \in \Omega^{n-1}(A_\delta)$ such that $\omega_\delta$ interpolates $*_{g_\lambda} \d u_\lambda$ and $*_g \d f_q$ while $v_\delta \wedge \omega_\delta > 0$ on $A_\delta$. For this consider $*_g \d f_q - *_{g_\lambda} \d u_\lambda$. With respect to the metric $g_\lambda$,
\begin{equation}
  \label{eq:dual-form-estimate-1}
  *_g \d f_q - *_{g_\lambda} \d u_\lambda = \mathcal{O}(\delta^{m-1} (\delta + R^{-1})).
\end{equation}
Viewing $A_\delta = (\delta/2, \delta) \times S^{n-1}$, denote 
\begin{equation}
  \label{eq:dual-form-estimate-2}
  *_g \d f_q - *_{g_\lambda} \d u_\lambda = \d r \wedge \sigma + \tau
\end{equation}
for some unique $\sigma, \tau \in C^\infty(A_\delta) \otimes \Omega^{\bullet}(S^{n-1})$.
Using the standard homotopy operator for differential forms \cite[Chapter~I, \S4]{BottTu1982}, one can show that
\begin{equation}
  \label{eq:dual-form-estimate-3}
  *_g \d f_q - *_{g_\lambda} \d u_\lambda = \d \left( \int_{\frac{3}{4}\delta}^r \sigma \right) + \tau|_{S^{n-1}_{3 \delta/4}},
\end{equation}
where $S^{n-1}_{3 \delta/4} = \{\frac{3}{4} \delta \} \times S^{n-1} \subset A_\delta$ is the sphere of radius $\frac{3}{4} \delta$.
We claim that $\tau|_{S^{n-1}_{3 \delta/4}}$ is exact. Indeed, the group $H_{n-1}(S^{n-1};\RR)$ is generated by a sphere,
so $\tau|_{S^{n-1}_{3 \delta/4}}$ is exact if its integral over $S^{n-1}_{3 \delta/4}$ vanishes.
Notice that $\int_{S^{n-1}_{3 \delta/4}} \tau = \int_{S^{n-1}_{3 \delta/4}} (*_g \d f_q - *_{g_\lambda} \d u_\lambda)$. The integral $\int_{S^{n-1}_{3 \delta/4}} *_{g_\lambda} \d u_\lambda$ vanishes, because the Euclidean model has zero flux. The integral $\int_{S^{n-1}_{3 \delta/4}} *_{g} \d f_q$ vanishes by Stokes' theorem. Thus, $\tau|_{S^{n-1}_{3 \delta/4}}$ is exact.

To pick a primitive of $\tau|_{S^{n-1}_{3 \delta/4}}$, let $G$ be the Greens operator of the Hodge Laplacian on the $(n-1)$-sphere of radius $1$. By the Hodge-decomposition theorem, $\tau|_{S^{n-1}_{3 \delta/4}} = \Delta G \tau|_{S^{n-1}_{3 \delta/4}} = \d \d^{\tilde{*}} G\: \tau|_{S^{n-1}_{3 \delta/4}}$, where $\tilde *$ is the Hodge dual with respect to the standard round metric. 
Let $\xi_\delta = \d^{\tilde{*}} G\: \tau|_{S^{n-1}_{3 \delta/4}}$. By taking into account the difference between the standard round metric and the metric on $S^{n-1}_{3 \delta/4}$, elliptic regularity gives
\begin{equation}
  \label{eq:dual-form-estimate-4}
\| \xi_\delta \|_{C^0(S^{n-1}_{3 \delta/4})} \le C \:\delta\: \| \tau\|_{C^0(S^{n-1}_{3 \delta/4})},
\end{equation}
for some constant $C > 0$.

In summary, we get that $\eta_\delta := \int_{\frac{3}{4}\delta}^r \sigma + \xi_\delta$ satisfies
$
\d \eta_\delta = *_g \d f_q - *_{g_\lambda} \d u_\lambda, 
$
and \eqref{eq:dual-form-estimate-1}, \eqref{eq:dual-form-estimate-2}, \eqref{eq:dual-form-estimate-3}, and \eqref{eq:dual-form-estimate-4} imply
$
\eta_\delta = \mathcal{O}_{g_\lambda}(\delta^{m} (\delta + R^{-1})).
$

Finally, define
\[
\begin{aligned}
\omega_\delta
&:=*_{g_\lambda}\d u_\lambda+\d(\chi\eta_\delta)\\
&=\chi *_g\d f_q
+(1-\chi)*_{g_\lambda}\d u_\lambda
+\d\chi\wedge\eta_\delta.
\end{aligned}
\]
The form $\omega_\delta$ is closed.  It equals $*_g v$ near the outer boundary
and $*_{g_\lambda} \d u_\lambda$ near the inner boundary of $A_\delta$. It also satisfies
$\omega_\delta-*_{g_\lambda} \d P_q = \mathcal{O}_{g_\lambda}(\delta^{m-1}(\delta+R^{-1}))$ with respect to $g_\lambda$. 

\medskip
\noindent
\textbf{Gluing conditions:}
We will apply Proposition \ref{prop:calabi-metric-realization} on $A_\delta$ using $v_\delta$ and $\omega_\delta$. Recall that on $A_\delta$, the 1-form $v_\delta$ equals $\d P_q + \mathcal{O} (\delta^{m-1}(\delta+R^{-1}))$. The matching condition gives $|\d P_q(x)|\geq c|x|^{m-1} = \mathcal{O}(\delta^{m-1})$ and so $v_\delta$ is nowhere vanishing on $A_\delta$ when $\delta$ is sufficiently small and $R$ is sufficiently large.
Similarly, $v_\delta\wedge \omega_\delta > 0$ on $A_\delta$, because its leading term is $\d P_q\wedge\star_{g_0}\d P_q=|\d P_q|_{g_0}^2 \operatorname{vol}_{g_0}$.
Finally, notice that $\ker v_\delta$ is a vector subbundle of $TA_\delta$ and that both $g$ and $g_\lambda$ induce a metric on this bundle.
Let $h_\delta$ be a smooth interpolation of these metrics on $\ker \alpha$. According to Proposition \ref{prop:calabi-metric-realization}, the metric $g_\delta := f^2 v_\delta^2 + \pi^* h_\delta$ is a Riemannian metric on $A_\delta$ such that $*_{g_\delta} v_\delta = \omega_\delta$ and $\pi^* \operatorname{vol}_{h_\delta} = f \cdot \beta$. Moreover, this metric smoothly extends to $g_\lambda$ on $B_{\delta/2}(0) \subset \RR^n$. It also extends smoothly to $g$ on $M \setminus B_{\delta}(q)$. By extending $g_\delta$ to the two complementary regions, we conclude $(v_\delta,\Sig_\delta,\I_\delta)$ is a $\ZT$-harmonic $1$-form on $(M,g_\delta)$.
\end{proof}

\subsection{Applications}
\label{subsec:zero-flux-applications}
The gluing theorem in the previous subsection leaves two tasks: First, we must construct a zero-flux Euclidean model with a suitable leading polynomial, and secondly we find a matching point for that polynomial on the closed manifold. 

In this section we start with the linear case, as this is the easier case. After this, we turn to a higher-degree construction in dimension three.

\begin{corollary}
\label{cor:linear-zero-flux-insertion}
Let $(M^n,g)$ be a closed oriented Riemannian manifold, where $n\geq 3$, and
suppose that it carries a non-zero $\ZT$-harmonic $1$-form
$(v,\Sig_M,\I_M)$.  Let $B\subset M \setminus \Sig_M$ be an open ball.  Then every
zero-flux Euclidean model whose asymptotic polynomial has degree one can be
inserted into $B$.  The new metric, $1$-form, and flat line bundle agree with
the original data on $M\setminus B$.

In particular, if $b_1(M)>0$, then this conclusion holds for every Riemannian
metric $g$ on $M$ and every open ball $B\subset M$.
\end{corollary}

\begin{proof}
Since $v$ is non-zero, unique continuation gives a regular point $q\in B$.
Because of this, the matching polynomial at $q$ is non-zero and linear.
In degree one, there is only one homogeneous linear polynomial up to scaling and rotation, and so every degree-one
zero-flux Euclidean model will match it. 
So we can apply Theorem~\ref{theorem:calabi-insertion} inside a small geodesic ball $B_\delta(q)$ around $q$.

If $b_1(M)>0$, Hodge theory gives a non-zero ordinary harmonic $1$-form for
every metric $g$.  Regarding it as a $\ZT$-harmonic $1$-form with empty
singular set and trivial line bundle, we can apply the first assertion to it.
\end{proof}

\begin{proof}[Proof of Theorem \ref{corollary:b1-positive}]
Choose a $4$-dimensional subspace $V$ of the harmonic polynomials on $\RR^n$
of degree at most one.  Since $\Sigma_0$ is connected,
Theorem~\ref{thm:flux-normalization} gives $P_s\in V$ and a small normal
deformation $\nu_s$ such that $u_{P_s,s,\nu_s}$ is nondegenerate,
$\ZT$-harmonic, and has zero flux.

We claim that the polynomial $P_s$ cannot be constant, as this would imply $\|\d u_{P_s,s,\nu_s} \|_{L^2} = 0$.
Namely, integrating by parts gives 
$$\|\d u_{P_s,s,\nu_s} \|_{L^2}^2 = \cancel{\langle u_{P_s,s,\nu_s}, \Delta u_{P_s,s,\nu_s} \rangle_{L^2}} + \int_{\partial} u_{P_s,s,\nu_s} * \d u_{P_s,s,\nu_s}.$$
If $P_s$ is constant, then the zero-flux at infinity and the order-$3/2$ vanishing along $\Sigma_{\nu_s}$ make the boundary term vanish.
This would imply $u$ is constant everywhere, which must be zero due to the monodromy condition.
As $u \not=0$, we get that $P_s$ has degree one.

Corollary~\ref{cor:linear-zero-flux-insertion} inserts this model into $B$.
The inserted singular set is a scaled copy of $\Sigma_{\nu_s}$, which is a
small normal deformation of $\Sigma_0$, and the new metric agrees with $g$ on
$M\setminus B$.
\end{proof}

If $\Sigma_0$ has $p$ connected components,
Theorem~\ref{thm:flux-normalization} requires a $4p$-dimensional polynomial
space $V$.  The harmonic polynomials of degree at most one, form an
$(n+1)$-dimensional space, so they contain such a space $V$ only when
$4p\leq n+1$.  Accommodating more components within the same construction
therefore requires higher-degree asymptotic polynomials and matching points
of correspondingly higher degree.

In dimension three, zonal harmonic polynomials provide a class of
higher-degree matching polynomials.  Let $P_m$ be the $m$-th Legendre
polynomial and set
\[
 Z_m(x,y,z)=r^mP_m(z/r),\qquad r=(x^2+y^2+z^2)^{1/2}.
\]
The polynomial $Z_m$ is homogeneous and harmonic and is invariant under
rotations about the $z$-axis.

It also satisfies the nonvanishing condition of a matching polynomial.
Indeed, suppose that $\d Z_m=0$ at a non-zero point $p \in \RR^n$. 
Because $Z_m$ is homogeneous, differentiation along $\gamma(s) = s \cdot p$ gives $\d Z_m (\dot{\gamma}(s)) = m s^{-1}Z_m(\gamma(s))$.
Hence, $\d Z_m|_p = 0$ implies $Z_m(p) = 0$. Notice that $p$ cannot lie on the $z$-axis, because on the $z$-axis $\frac{z}{r} = \pm 1$ while $P_m(\pm1)\neq0$.
Differentiating $Z_m$ gives
\begin{align*}
  \d Z_m 
  =& r^{m-2} (m \: r P_m(z/r) - z P'_m(z/r)) \d r 
  + r^{m-1} P'_m(z/r) \d z.
\end{align*}
Except on the $z$-axis, $\d r$ and $\d z$ are linearly independent 1-forms. Hence, if $\d Z_m (p) = 0$, then both
$P_m(z/r)$ and $P'_m(z/r)$ vanish at $p$. This contradicts the simplicity of the roots of $P_m$. Therefore, $Z_m$ is a matching polynomial.

Since the polynomials $Z_3,\ldots,Z_{4p+2}$ have distinct degrees, their
span $V$ has dimension $4p$. If $P\in V\setminus \{0\}$, then its top degree homogeneous term of $P$ is a nonzero, non-constant zonal harmonic and hence satisfies the matching condition.  Applying
Theorem~\ref{thm:flux-normalization} with this choice of $V$ gives the
following Euclidean models:

\begin{corollary}
\label{cor:zonal-zero-flux-model-existence}
Suppose that $\Sigma_0\subset\RR^3$ is a compact oriented codimension-two
submanifold with trivial normal bundle and $p$ connected components.  For
sufficiently large $s$, there is an integer
$m_s\in\{3,\ldots,4p+2\}$, a harmonic polynomial
$P_s\in\operatorname{span}\{Z_3,\ldots,Z_{4p+2}\}$ of degree $m_s$, and a
normal deformation $\nu_s$ of $\Sigma_0$, such that the function
$u_{P_s,s,\nu_s}$ is nondegenerate and $\ZT$-harmonic. Moreover, $u_{P_s,s,\nu_s}$ has zero flux, and is
asymptotic to $P_s$ at infinity.  The highest-degree homogeneous term of
$P_s$ is a nonzero multiple of $Z_{m_s}$.
\end{corollary}

The only non-obvious part of the proof of this corollary is that $P_s \not = 0$.
However, this follows from the same argument in the proof of Theorem \ref{corollary:b1-positive} that $P_s$ cannot be constant.
 

This completes the Euclidean part of the construction.  The compact-side
input is provided by the following result.

\begin{proposition}[{\protect\cite[Theorem~4.5]{HeChenYan}}]
\label{prop:zonal-matching-point-preparation}
Let $(v,\Sig_M,\I_M)$ be a $\ZT$-harmonic $1$-form on $(M^3,g)$, and let $q$ be a regular point of $v$.  Fix
$m\geq3$ and an open ball $B$ about $q$.  There is a metric $g'$ and a
$\ZT$-harmonic $1$-form $(v',\Sig_M,\I_M)$ on $(M,g')$ such that $v'=v$
on $M\setminus B$ and $q$ is a matching point for $v'$ with matching
polynomial $Z_m$.
\end{proposition}

The preceding corollary and proposition provide the two inputs to the gluing
theorem.

\begin{corollary}
\label{cor:zonal-zero-flux-insertion}
Let $(M^3,g)$ be a closed oriented Riemannian manifold, and suppose that it
carries a nonzero $\ZT$-harmonic $1$-form $(v,\Sig_M,\I_M)$.  Then every
zero-flux Euclidean model $u_{P_s,s,\nu_s}$ furnished by
Corollary~\ref{cor:zonal-zero-flux-model-existence} can be inserted in any
prescribed open ball $B$ about any regular point $q$ of $v$.  The resulting
$1$-form agrees with $v$ on $M\setminus B$, and its singular set is obtained
from $\Sig_M$ by adjoining a scaled copy of $\Sigma_{\nu_s}$ that collapses
to $q$ as the insertion scale tends to zero.

In particular, if $b_1(M)>0$, the initial metric $g$ may be chosen
arbitrarily.
\end{corollary}

\begin{remark}
Unlike the degree-one insertion in
Corollary~\ref{cor:linear-zero-flux-insertion}, the matching-point preparation
used above gives $v'=v$ on $M\setminus B$, but provides no corresponding
relation between $g'$ and $g$ there.  Consequently, the newly inserted
singular set and the change in the $1$-form are confined to $B$, while the
final metric is not controlled outside $B$.
\end{remark}

\bibliographystyle{alpha}
	\bibliography{references}
\end{document}